\documentclass[11pt]{amsart}
\usepackage{amsmath}
\usepackage{amssymb}
\usepackage{amsfonts}
\usepackage{amsthm}
\usepackage{enumerate}
\usepackage[mathscr]{eucal}
\usepackage{graphicx}
\usepackage[pdftex, bookmarksnumbered, bookmarksopen, colorlinks, citecolor=blue, linkcolor=blue]{hyperref}

\newtheorem{theorem}{Theorem}[section]
\newtheorem{lemma}[theorem]{Lemma}

\theoremstyle{definition}
\newtheorem{definition}[theorem]{Definition}

\theoremstyle{remark}
\newtheorem{remark}[theorem]{Remark}
\newtheorem*{remark*}{Remark}

\numberwithin{equation}{section}

\def\intslash{\rlap{\kern  .32em $\mspace {.5mu}\backslash$ }\int}
\def\qsl{{\rlap{\kern  .32em $\mspace {.5mu}\backslash$ }\int_{Q_x}}}

\def\off{{\text{\rm off}}}

\def\R{{\mathbb R}}
\def\Z{{\mathbb Z}}
\def\M{{\mathcal M}}
\def\C{{\mathcal C}}

\def\A{{\mathcal A}}
\def\Q{{\mathcal Q}}
\def\esssup{\mathop{\mathrm{ess \ sup}}}

\def\ot{\overline{\otimes}}

\def\S{\mathbf S}
\def\Q{\mathcal Q}

\def\pari{\partial}
\def\Ga{\Gamma}
\def\ga{\gamma}
\def\Th{\Theta}

\def\eg{{\it e.g. }}

\def\supp{{\text{\rm supp}}}

\def\inn#1#2{\langle#1,#2\rangle}

\def\rta{\rightarrow}

\def\card{\text{\rm card}}
\def\lc{\lesssim}

\def\pv{\text{\rm p.v.}}
\def\alp{\alpha}

\def\del{\delta}             
\def\eps{\varepsilon}

\def\zet{\zeta}
\def\tet{\theta}

\def\ka{\kappa}
\def\lam{\lambda}            \def\Lam{\Lambda}

\def\si{\sigma}              
\def\vphi{\varphi}
             
\def\om{\omega}              \def\Om{\Omega}
\def\fr{\frac}

\newcommand{\Be}{\begin{equation}}
\newcommand{\Ee}{\end{equation}}
\newcommand{\Bes}{\begin{equation*}}
\newcommand{\Ees}{\end{equation*}}
\newcommand{\Bsp}{\begin{split}}
\newcommand{\Esp}{\end{split}}
\newcommand{\Bm}{\begin{multline}}
\newcommand{\Em}{\end{multline}}
\newcommand{\Bea}{\begin{eqnarray}}
\newcommand{\Eea}{\end{eqnarray}}
\newcommand{\Beas}{\begin{eqnarray*}}
\newcommand{\Eeas}{\end{eqnarray*}}
\newcommand{\Benu}{\begin{enumerate}}
\newcommand{\Eenu}{\end{enumerate}}
\newcommand{\Bi}{\begin{itemize}}
\newcommand{\Ei}{\end{itemize}}
\begin{document}

\title[Noncommutative maximal operators]{Noncommutative maximal operators with rough kernels}

\author{Xudong Lai}
\address{Xudong Lai: Institute for Advanced Study in Mathematics, Harbin Institute of Technology, Harbin, 150001, People's Republic of China}
\email{xudonglai@hit.edu.cn}
\thanks{This work was supported by National Natural Science Foundation of China (No. 11801118), China Postdoctoral Science Foundation (No. 2017M621253, No. 2018T110279)  and  Fundamental Research Funds for the Central Universities (No. FRFCU5710050121).}





\begin{abstract}
This paper is devoted to the study of noncommutative maximal operators with rough kernels. More precisely, we prove the weak type $(1,1)$ boundedness for noncommutative maximal operators with rough kernels. The proof of weak type (1,1) estimate is based on the noncommutative Calder\'on-Zygmund decomposition.
To deal with the rough kernel, we use the microlocal decomposition in the proofs of both the bad and good functions.
\end{abstract}

\maketitle

\section{Introduction and state of main result}
In recent years, there are extensive research on the noncommutative harmonic analysis, especially on the noncommutative Calder\'on-Zygmund theory (see \eg \cite{Par09}, \cite{MP09}, \cite{Cad18}, \cite{CXY13}).
The main content of this topic is focused on investigating the boundedness property of various operators in harmonic analysis on the noncommutative $L_p$ space.
Due to the lack of commutativity (i.e. $ab=ba$ may not hold in general case), many problems in the study of noncommutative Calder\'on-Zygmund theory seem to be more difficult, for instance the weak type (1,1) bound of integral operator.

It is well known that the real variable theory of classical harmonic analysis is initialed by A. P. Calder\'on and A. Zygmund in 1950s (see \cite{CZ52}). One of the remarkable technique in \cite{CZ52} is the so-called Calder\'on-Zygmund decomposition, which is now a widely used method in harmonic analysis.
This technique not only gives a real variable method to show weak type (1,1) bounds of singular integrals, but also provides a basic idea of stopping time arguments for many topics in harmonic analysis, such as the theory of Hardy and BMO spaces (see \eg \cite{Gra14C}, \cite{Gra14M}, \cite{Ste93}).
The noncommutative Calder\'on-Zygmund decomposition was recently established in \cite{Par09} via the theory of noncommutative martingales.
With this tool,  the weak type (1,1) bound theory of the standard Calder\'on-Zygmund operator was developed there. It was pointed out in \cite{Par09} that the noncommutative Calder\'on-Zygmund decomposition and the related method should open a door to work for more general class of operators. For the subsequent works related to weak type (1,1) problem and noncommutative Calder\'on-Zygmund decomposition, we refer to see \cite{MP09}, \cite{Cad18}, \cite{CPSZ19}, \cite{CSZ18}, \cite{HX21}, \cite{HLX22}, \cite{CCP22} and the references therein.

On the other hand, the classical theory of singular integral operator tells us that to ensure the weak type (1,1) bound of the Calder\'on-Zygmund operator, the regularity condition of kernel can be relaxed to the so-called H\"ormander condition (see \eg \cite{Hor60} or \cite{Gra14C}). Moreover, A. P. Calder\'on and A. Zygmund in \cite{CZ56} further studied the singular integral operator with a rough homogeneous kernel defined by
\Be\label{e:19Tom}
\pv\int_{\R^d}\fr{\Om(x-y)}{|x-y|^d}f(y)dy
\Ee
and established its $L_p$ boundedness for all $1<p<\infty$. For its weak type (1,1) boundedness, it was quite later showed by S. Hofmann \cite{Hof88} (independently by M. Christ and Rubio de Francia \cite{CR88}) in two dimensions and by  A. Seeger in higher dimensions \cite{See96} (see further result by T. Tao \cite{Tao99}). Therefore a natural question inspired by work \cite{Par09} is that can we weaken the Lipschitz regularity of  kernel to the H\"ormander condition or even rough homogeneous kernel. This problem has been open since then. The purpose of this paper is to develop some theory in this aspect for a class of rough operators.
We consider the most fundamental operator: the   maximal operator with a rough kernel which is defined by (in the sense of classical harmonic analysis)
\Be\label{e:19mom}
M_\Om f(x)=\sup_{r>0}|M_rf(x)|,\quad M_rf(x)=\fr{1}{|B(x,r)|}\int_{B(x,r)}\Om(x-y)f(y)dy
\Ee
where $B(x,r)$ is a ball in $\R^d$ with center $x$ and radius $r$, the kernel $\Om$ is  a  homogeneous function defined on $\R^d\setminus\{0\}$ with degree zero, that is
\begin{equation}\label{e:19Hom}
\Om(rx')=\Om(x'),\text{ for any $r>0$ and $x'\in\S^{d-1}$}.
\end{equation}

Notice that the maximal operator $M_\Om$ is a generalization of the Hardy-Littlewood maximal operator (by setting $\Om$ as a constant then $M_\Om$ is exactly the Hardy-Littlewood maximal operator). $M_\Om$ is very important in the theory of rough singular integral since it could be used to control many operators with rough kernels, which just like that the Hardy-Littlewood maximal operator plays an important role in analysis. By the method of rotation, it is easy to see that $M_\Om$ is bounded on $L_p(\R^d)$ for all $1<p\leq\infty$ if $\Om\in L_1(\S^{d-1})$ (see \eg \cite{Gra14C}).
However, the weak (1,1) boundedness of $M_\Om$ is quite challenging.
It was proved by M. Christ \cite{Chr88} that $M_\Om$ is of weak type (1,1) if $\Om\in L_q(\S^1)$ with $1<q\leq\infty$ in two dimensions. Later M. Christ and Rubio de Francia \cite{CR88} showed in higher dimensions $M_\Om$ is weak (1,1) bounded if $\Om\in L\log^+ L(\S^{d-1})$ by a depth investigation of the geometry in Euclidean space.
For more topics including open problems related to the maximal operator $M_\Om$, we refer to see the survey article \cite{Ste98}, \cite{GS99}, \cite{GSP17} and the references therein.

The noncommutative version of $M_\Om$ should be important in the theory of noncommutative rough singular integral operators as expected. For instance, the noncommutative $M_\Om$ will play a crucial role in the study of the noncommutative maximal operator of truncated operator in \eqref{e:19Tom}. In this paper, we will study the boundedness of $M_\Om$ on the noncommutative $L_p$ space for $1\leq p\leq \infty$. In a special case $\Om$ is a constant (i.e. $M_\Om$ is the Hardy-Littlewood maximal operator), T. Mei \cite{Mei07} has investigated its noncommutative $L_p(1<p\leq\infty)$ and weak type $(1,1)$ boundedness. For general kernel $\Om$, there is no proper theory for the noncommutative $M_\Om$. To illustrate our noncommutative result of $M_\Om$, we should first give some basic notation.

Let us first introduce the noncommutative $L_p$ space.
Let $\M$ be a semi-finite von Neumann algebra equipped with a normal semifinite faithful (\emph{n.s.f.} in short) trace $\tau$. We consider the algebra $\A_B$ of essentially bounded $\M$-valued function
\Bes
\A_B=\big\{f:\R^d\rta\M\ |\ f\ \text{is strong measurable such that}\ \esssup_{x\in\R^d}\|f(x)\|_\M<\infty.\big\}
\Ees
equipped with the \emph{n.s.f.} trace $\vphi(f)=\int_{\R^d}\tau(f(x))dx$. Let $\A$ be the weak-operator closure of $\A_B$. Then $\A$ is a von Neumann algebra.
For $1\leq p<\infty$, define $L_p(\M)$ as the noncommutative
$L_p$ space associated to the pairs
$(\M,\tau)$ with the $L_p$ norm given by $\|x\|_{L_p(\M)}=(\tau(|x|^p))^{1/p}$.
The space
$L_p(\mathcal{A})$ is defined as the closure of $\mathcal{A}_B$
with respect to the following norm
\Be\label{e:19deflp}
\|f\|_{L_p(\A)}= \Big( \int_{\R^d} \tau \,
\big( |f(x)|^p \big) \, dx \Big)^{\frac1p},
\Ee
which means that $L_p(\A)$ is the noncommutative
$L_p$ space associated to the pairs
$(\A,\varphi)$. On the other hand, from \eqref{e:19deflp} we see that
$L_p(\mathcal{A})$ is isometric to the Bochner $L_p$ space with
values in $L_p(\mathcal{M})$. For convenience, we set $L_\infty(\M)=\M$ and $L_\infty(\A)=\A$ equipped with the operator
norm.
The lattices of projections are written as $\M_\pi$ and $\A_{\pi}$,
while $1_{\M}$ and $1_{\A}$ stand for the unit elements. Let $L_p^+(\A)$ be the positive part of $L_p(\A)$.
A lot of basic properties of classical $L_p$ spaces, such as Minkowski's inequality, H\"older's inequality, dual property, real and complex interpolation, etc,  have been transferred to
this noncommutative setting.
We refer to the very detailed introduction in \cite{Par09} or the survey article \cite{PX03} for more about the noncommutative $L_p$ space, the noncommutative $L_{1,\infty}$ space and related topics.

We next define a noncommutative analogue of $M_\Om$. For two general elements belong to a von Neumann algebra, they may not be comparable (i.e. neither $a<b$ nor $a\geq b$ holds for $a,b\in\A$). Hence it is difficult to define the noncommutative maximal function directly. This obstacle could be overcome by straightforwardly define the maximal  weak type (1,1) norm or $L_p$ norm.
We adopt the definition of the noncommutative maximal norm
introduced by G. Pisier \cite{Pis98} and M. Junge \cite{Jun02}.
\begin{definition}
For any index set $I$, we define
$L_p(\M;\ell_\infty(I))$ the space of all sequences
$x=\{x_n\}_{n\in I}$ in $L_p(\M)$ which admits a factorization of the
following form: there exist $a, b\in L_{2p}(\M)$ and a bounded
sequence $y=\{y_n\}_{n\in I}$ in $L_\infty(\M)$ such that
 $x_n=ay_nb$, $\forall\; n\in I$.
The norm of  $x$ in $L_p(\M;\ell_\infty(I))$ is given by
 $$\|\{x_k\}_{k\in I}\|_{L_p(\M;\ell_\infty(I))} =\inf\big\{\|a\|_{L_{2p}(\M)}\,
 \sup_{n\in I}\|y_n\|_{L_\infty(\M)}\,\|b\|_{L_{2p}(\M)}\big\} ,$$
where the infimum is taken over all factorizations of $x$ as above. We define a sequences $x=\{x_k\}_{k\in I}$ in $L_{1,\infty}(\M)$ with its quasi-norm given by
\begin{align*}
\|\{x_k\}_{k\in I}\|_{\Lambda_{1,\infty}(\M;\ell_{\infty}(I))}
=\sup_{\lambda>0}\lambda\inf_{e\in\M_\pi}
\Big\{\tau(e^{\perp}):\|ex_ke\|_{\infty}\leq\lambda\ \mbox{for\ all}\ k\in I\Big\}.
\end{align*}
\end{definition}

If $x=\{x_n\}_{n\in I}$ is a sequence of positive elements, then $x\in L_p(\M;\ell_\infty(I))$ if and only if there exists a positive element $a\in L_p(\M)$ such that $0<x_n\leq a$, and
\Be\label{e:19positivem}
\|x\|_{L_p(\M;\ell_\infty(I))}=\inf\{\|a\|_{L_p(\M)}:\ 0<x_n\leq a, \forall n\in I\};
\Ee
\begin{align}\label{e:19weakm11}
\|(x_k)_{k\in I}\|_{\Lambda_{1,\infty}(\M;\ell_{\infty}(I))}=\sup_{\lambda>0}\lambda\inf_{e\in\M_\pi}
\Big\{\tau(e^{\perp}): ex_ne\leq\lambda\ \mbox{for\ all}\ n\Big\}.
\end{align}

Now we can state our main result as follows.

\begin{theorem}\label{t:19}
Suppose that $\Om$ satisfies \eqref{e:19Hom} and $\Om\in L(\log^+ L)^{2}(\S^{d-1})$. Then the operator sequences $\{M_r\}_{r>0}$ is of maximal weak type (1,1), i.e.
\Bes
\|\{M_rf\}_{r>0}\|_{\Lambda_{1,\infty}(\A,\ell_\infty(0,\infty))}\lc \C_\Om\|f\|_{L_1(\A)},
\Ees
where $\C_\Om$ is a constant depends only on the dimension $d$ and $\Om$.
Equivalently, for every $f\in L_1(\A)$ and $\lam>0$, there exists a projection $e\in\A_\pi$ such that
    \Bes
\sup_{r>0}\|eM_rfe\|_{L_\infty(\A)}\leq\lam \ \ \text{and}\ \ \ \lam\vphi(1_\A-e)\lc\C_\Om\|f\|_{L_1(\A)}.
    \Ees

\end{theorem}

It is very easy to show that $\{M_r\}_{r>0}$ is of maximal strong type $(p,p)$ for $1<p\leq \infty$ by the method of rotation. For completeness, we give a proof in the appendix for this result.

The strategy in the proof of Theorem \ref{t:19} is as follows. Firstly we convert the study of the maximal operator to a linearized singular integral operator $T$ with a rough kernel (see Section \ref{s:193}).
Secondly, to prove the weak type (1,1) bound of this singular integral operator $T$, we use the noncommutative Calder\'on-Zygmund decomposition to split the function $f$ as two parts: good functions and bad functions (see Subsection \ref{s:1941}).
Roughly speaking, the proof is reduced to obtain some decay estimates for the good and bad functions separately.
For the proof related to the bad functions, since the kernel $\Om$ is rough, we will use a further decomposition, the so-called microlocal decomposition, to the operator $T_j$. Then we apply the $L_2$ norm and the $L_1$ norm to control the weak type estimate (see Lemma \ref{l:19L^2} and Lemma \ref{l:19L^1}), where vector valued Plancherel's theorem, orthogonal argument of geometric are involved in the proof of the $L_2$ estimate and the stationary phase method is used in the $L_1$ estimate (see Subsection \ref{s:1942} and the proofs in Section \ref{s:195}).
For the proof of good functions we use the so-called pseudo-localisation arguments to obtain some decay estimate for the $L_2$ norm of singular integral operator $T$ outside the support of functions on which it acts. To get such decay estimates, we adopt a similar method (the microlocal decomposition) from the proof of bad functions (see Subsection \ref{s:1943}).

In the classical Calder\'on-Zygmund decomposition, one can easily deal with the good function by the $L_2$ estimate. However the proof of good functions from the noncommutative Calder\'on-Zygmund is much elaborated as showed in the case of smooth kernel by J. Parcet \cite{Par09}. In this paper, to overcome the non smoothness of  kernel, we use the microlocal decomposition in the proofs of both bad and good functions. To the best knowledge of the author, this method seems to be new in the noncommutative Calder\'on-Zygmund theory.
We should point out that  the proof of  bad functions is quite different from that in the classical case by M. Christ and Rubio de Francia \cite{CR88}, where they  used the $TT^*$ argument to obtain some regularity of the kernel $T_jT_j^*$ by some depth geometry but without using the Fourier transform.
However our method presented in this paper heavily depends on the Fourier transform where Plancherel's theorem and the stationary phase method are involved. These ideas are mainly inspired by A. Seeger's work \cite{See96}.  Recall  the following important pointwise property is crucial in classical $TT^*$ argument:
$
|Q|^{-1}\int_Q|b_Q(y)|dy\lc\lam,
$
where $b_Q$ is a bad function from the Calder\'on-Zygmund decomposition which is supported in a cube $Q$.
Since in the noncommutative setting such kind of inequalities may not hold for the off-diagonal terms of bad functions,
our noncommutative $TT^*$ argument is more complicated than that of the classical case.
In fact only one pointwise property holds in the noncommutative Calder\'on-Zygmund decomposition: $q_kf_kq_k\lc\C^{-1}_\Om\lam q_k$ (see Lemma \ref{l:19cucu}) and all pointwise estimate in the proof should finally be transferred to this property (see Subsection \ref{s:1952} for the details in the proof of the $L_2$ estimate).

This paper is organized as follows. First the study for maximal operator of $M_r$ is reduced to a linearized singular integral operator in Section \ref{s:193}.
In Section \ref{s:194}, by the noncommutative Calder\'on-Zygmund decomposition and microlocal decomposition, we finish the proof of our main theorem based on the estimates of bad and good functions. The proofs of lemmas related to the bad functions are all presented in Section \ref{s:195}. In Section \ref{s:196}, we give all proofs of lemmas related to the good functions.
Finally in the appendix, we give a proof of strong type $(p,p)(1<p\leq\infty)$ for $\{M_r\}_{r>0}$.

\subsection*{Further remark} After we finish this manuscript, L. Cadilhac found a more efficient noncommutative Calder\'on-Zygmund decomposition (see \cite{HLX22} and \cite{CCP22}) so that the off-diagonal terms of the good functions vanish and the argument for the pseudo-localisation can be avoided.
Of course using this new Calder\'on-Zygmund decomposition, we only need to apply the $L_2$ estimate to deal with the good function and the proof related to the good functions in this paper can be greatly shorten. However we point out that using this new method, the proof for the bad functions will be significantly more complicated than our arguments presented in this paper. So our proof in this paper still has its own interest. Nevertheless, we hope to show this in the study of weak type $(1,1)$ boundedness for singular integral operators with rough kernels \eqref{e:19Tom} which is our ongoing work.

\subsection*{Notation} Throughout this paper, we only consider the dimension $d\ge2$ and the letter $C$ stands for a positive finite constant which is independent of the essential variables, not necessarily the same one in each occurrence. $A\lc B$ means $A\leq CB$ for some constant $C$. By the notation $C_\eps$  we mean that the constant depends on the parameter $\eps$. $A\approx B$ means that $A\lc B$ and $B\lc A$.
$\Z_+$ denotes the set of all nonnegative integers and $\Z_+^d=\underbrace{\Z_+\times \cdots\times \Z_+}_d.$ For $\alp\in\Z_+^d$ and $x\in\R^d$, we define $x^\alp=x_1^{\alp_1}x_2^{\alp_2}\cdots x_d^{\alp_d}$ and $|x|$ denotes the $\ell_2$ norm. $\forall s\in\R_+$, $[s]$ denotes the integer part of $s$.
For any set $A$ with finite elements, we define $\card(A)$ or $\#(A)$ as the number of elements in $A$.
Let $s\geq0$, we define $$\|\Om\|_{L(\log^+\!\!L)^s}:=\int_{\S^{d-1}}|\Om(\tet)|[\log(2+|\Om(\tet)|)]^sd\si(\tet),$$
where $d\si(\tet)$ denotes the sphere measure of $\S^{d-1}$. When $s=0$, we use the standard notation $\|\Om\|_{1}:=\|\Om\|_{L(\log^+\!\!L)^0}$.

Define $\mathcal{F}f$ (or $\hat{f}$) and $\mathcal{F}^{-1}f$ (or $\check{f}$) the Fourier transform and the inverse Fourier transform
of $f$ by
$$\mathcal{F}f(\xi)=\int_{\R^d} e^{-i\inn{x}{\xi}}f(x)dx,\ \ \ \ \mathcal{F}^{-1}f(\xi)=\fr{1}{(2\pi)^{d}}\int_{\R^d}e^{i\inn{x}{\xi}}{f(x)dx}.$$

Let $\Q$ be the set of all dyadic cubes in $\R^d$. For any $Q\in\Q$, denote by $\ell(Q)$ the side length of the cube $Q$. Let $s Q$ be the cube with the same center of $Q$ such that $\ell(s Q)=s\ell(Q)$. Given an integer $k \in \Z$, $\Q_k$ will be defined as the set of dyadic cubes of side length $2^{-k}$. Let $|Q|$ be the volume of the cube $Q$. If $Q\in\Q$ and $f: \R^d \to
\M$ is integrable on $Q$, we define its average as
$f_Q = |Q|^{-1} \int_Q f(y) \, dy.$

For $k\in\Z$, set $\sigma_{k}$ as the $k$-th dyadic $\sigma$-algebra, i.e., $\sigma_{k}$ is generated by the dyadic cubes with side lengths equal to $2^{-k}$. Let $\mathsf{E}_k$ be the
conditional expectation which is associated to the classical dyadic
filtration $\sigma_{k}$ on $\R^d$. We also use $\mathsf{E}_k$ for the
tensor product $\mathsf{E}_k \otimes id_\M$ acting on $\A$. Then for  $1
\le p < \infty$ and $f \in L_p(\A)$, we get that
$$\mathsf{E}_k(f) = \sum_{Q \in \Q_k}^{\null} f_Q \chi_Q,$$
where $\chi_Q$ is the characteristic function of $Q$.
Similarly, $\{\A_k\}_{k \in \Z}$ will stand for the corresponding
filtration, i.e. $\A_k = \mathsf{E}_k(\A)$. For simplicity, we will write the conditional expectation $f_{k}:=\mathsf{E}_k(f)$ and the martingale difference $\Delta_{k} (f):=f_{k}-f_{k-1}=: df_{k}$.

\vskip0.24cm

\section{Reduction to singular integral operator}\label{s:193}
\vskip0.24cm
In this section, we reduce the study of maximal operator of $M_r$ to a singular integral operator with a rough kernel. This will be done by several steps as follows.

\emph{Step 1}.  By decomposing the functions $\Om$ and $f$ as four parts (i.e. real positive part, real negative part, imaginary positive part, imaginary negative part), together with the quasi-triangle inequality for the quasi-norm $\|\cdot\|_{\Lambda_{1,\infty}(\A,\ell_\infty(0,\infty))}$,
we only consider the case that $\Om$ is a positive function and $f$ is positive in $\A$. Then by \eqref{e:19weakm11}, it is enough to show that for any $f\in L_1^+(\A)$ and  $\lam>0$ there exists a projection $e\in\A_\pi$ such that
\Be\label{e:19weakdef}
eM_rfe\leq\lam ,\ \forall r>0 \ \ \text{and} \ \ \lam\vphi(1_\A-e)\lc\C_\Om\|f\|_{L_1(\A)}.
\Ee

\emph{Step 2}. Next we show that the study of $M_r$ can be reduced to a dyadic smooth operator. More precisely,
let $\phi$ be a $C_c^\infty(\R^d)$, radial, positive function which is supported in $\{x\in\R^d: 1/2\leq|x|\leq2\}$ and $\sum_{i\in\Z}\phi_j(x)=1$ for all $x\in\R^d\setminus\{0\}$, where $\phi_j(x)=\phi(2^{-j}x)$. Define an operator $\mathfrak{M}_j$ by
$$\mathfrak{M}_jf(x)=\int_{\R^d}\fr{\Om(x-y)}{|x-y|^d}\phi_j(x-y)f(y)dy.$$
We will prove that the maximal operator of $\mathfrak{M}_j$ is of weak type (1,1) below and \eqref{e:19weakdef} follows from it.

\begin{theorem}\label{t:19dya}
Let $\Om$ be a positive function satisfying \eqref{e:19Hom} and $\Om\in L(\log^+ L)^{2}(\S^{d-1})$.
For any $f\in L_1^+(\A)$, $\lam>0$, there exists a projection $e\in\A_\pi$ such that
    \Bes
    \sup_{j\in\Z}\|e\mathfrak{M}_jfe\|_{L_{\infty}(\A)}\lc\lam,\ \ \ \ \lam\vphi(1_\A-e)\lc\C_\Om\|f\|_{L_1(\A)},
    \Ees
where the constant $\C_\Om$ only depends on $\Om$ and the dimension.
\end{theorem}

The proof of Theorem \ref{t:19dya} will be given later. We apply Theorem \ref{t:19dya} to show \eqref{e:19weakdef}. Let $\Om$ be a positive function and $f$ be positive in $L_1^+(\A)$. Then by our choice of  $\phi_j$, for any $r>0$, we have
\Bes
\begin{split}
M_r f(x)&=\fr{1}{|B(x,r)|}\int_{B(x,r)}\Om(x-y)f(y)dy\\
&=\fr{C_d}{r^d}\sum_{j\leq[\log r]+1}\int_{|x-y|\leq r}\phi_j(x-y)\Om(x-y)f(y)dy\lc\fr{1}{r^d}\sum_{j\leq[\log r]+1}2^{jd}\mathfrak{M}_j f(x).
\end{split}
\Ees
Notice that $\Om$ is positive and  $f\in L_1^+(\A)$, the inequality $e\mathfrak{M}_jfe\leq\lam$ is equivalent to $\|e\mathfrak{M}_jfe\|_{\A}=\|e\mathfrak{M}_jfe\|_{L_\infty(\A)}\leq\lam$. By Theorem \ref{t:19dya}, there exists a projection $e\in\A_\pi$ such that
    \Bes
    e\mathfrak{M}_jfe\lc\lam,\ \ \forall j\in\Z, \ \ \lam\vphi(1_\A-e)\lc\C_\Om\|f\|_{L_1(\A)}.
    \Ees
Then it is easy to see that for any $r>0$,
\Bes
eM_rfe\lc\fr{1}{r^d}\sum_{j\leq[\log r]+1}2^{jd} e\mathfrak{M}_j f e \lc\lam.
\Ees

\emph{Step 3}. We will reduce the study of maximal operator of $\mathfrak{M}_j$ to a class of square function. Notice that the kernel $\Om$ of $\mathfrak{M}_j$ has no cancelation. Formally we can not study the operator $\Big(\sum_j|\mathfrak{M}_j|^2\Big)^{1/2}$ directly since it may be not even $L_2$ bounded. To avoid such case, we define a new operator $T_j$ which is modified version of the operator $\mathfrak{M}_j$
\Be\label{e:19Tj}
T_jf(x)=\int_{\R^d}\phi_j(x-y)\fr{\tilde{\Om}(x-y)}{|x-y|^{d}}f(y)dy,\quad \Ee
where $\tilde{\Om}(x)=\Om(x)-\fr{1}{\si_{d-1}}\int_{\S^{d-1}}\Om(\tet)d\si(\tet)$, $\si_{d-1}$ is measure of the unit sphere. Then it is easy to see that $\tilde{\Om}$ has mean value zero over $\S^{d-1}$.
Then formally the study of maximal operator of $\mathfrak{M}_j$ may follow from that of square function $\Big(\sum_j|T_j|^2\Big)^{1/2}$ and maximal operator. In the following we use rigorous noncommutative language to explain how to do it. To define noncommutative square function, we should first introduce the so-called column and row function space.
Let $\{f_{j}\}_j$ be a finite sequence in $L_{p}(\A) (1\leq p\leq\infty)$. Define
$$\|\{f_{j}\}_j\|_{L_{p}(\A; \ell_{2}^{r})}=\|(\sum|f^{\ast}_{j}|^{2})^{\frac{1}{2}}\|_{L_p(\A)},\ \|(f_{j})\|_{L_{p}(\A; \ell_{2}^{c})}=\|(\sum|f_{j}|^{2})^{\frac{1}{2}}\|_{L_p(\A)}.$$
This procedure is also used to define the spaces
$L_{1,\infty}(\A; \ell_{2}^{r})$ and $L_{1,\infty} (\A; \ell_{2}^{c})$, i.e.
$$\|\{f_{j}\}_j\|_{L_{1,\infty}(\A; \ell_{2}^{r})}=\|(\sum|f^{\ast}_{j}|^{2})^{\frac{1}{2}}\|_{L_{1,\infty}(\A)},\ \|\{f_{j}\}\|_{L_{1,\infty}(\A; \ell_{2}^{c})}=\|(\sum|f_{j}|^{2})^{\frac{1}{2}}\|_{L_{1,\infty}(\A)}.$$
Let $L_{1,\infty}(\A,\ell_2^{rc})$ space be the weak type square function of $\{T_j\}_j$ defined as
\Bes
\|\{T_j\}_j\|_{L_{1,\infty}(\A,\ell_2^{rc})}=\inf_{T_jf=g_j+h_j}\Big\{\|\{g_j\}_j\|_{L_{1,\infty}(\A,\ell_2^c)}
+\|\{h_j\}_j\|_{L_{1,\infty}(\A;\ell_2^r)}\Big\}.
\Ees
We have the following weak type (1,1) estimate of square function of $\{T_j\}_j$.
\begin{theorem}\label{t:19sq}
Suppose that $\Om$ satisfies \eqref{e:19Hom} and $\Om\in L(\log^+ L)^{2}(\S^{d-1})$.
Let $T_j$ be defined in \eqref{e:19Tj}, then we have
\Bes
\|\{T_j\}_j\|_{L_{1,\infty}(\A,\ell_2^{rc})}\lc\C_\Om\|f\|_{L_1(\A)}
\Ees
where the constant $\C_\Om$ only depends on $\Om$ and the dimension.
\end{theorem}

In the following we use Theorem \ref{t:19sq} to prove Theorem \ref{t:19dya}.
Our goal is to find a projection $e\in\A_\pi$ such that
    \Bes
    \sup_{j\in\Z}\|e\mathfrak{M}_jfe\|_{L_{\infty}(\A)}\lc\lam,\ \ \ \ \lam\vphi(1_\A-e)\lc\C_\Om\|f\|_{L_1(\A)}.
    \Ees
We first decompose $\mathfrak{M}_jf$ as two parts
\Bes T_jf(x)+\fr{1}{\si_{d-1}}\int_{\S^{d-1}}\Om(\tet)d\si(\tet)\int_{\R^d}\fr{\phi_j(x-y)}{|x-y|^d}f(y)dy=:T_jf(x)+\tilde{M}_jf(x).
\Ees
Notice that $\fr{1}{\si_{d-1}}\int_{\S^{d-1}}\Om(\tet)d\si(\tet)$ is a harmless constant which is bounded by $\|\Om\|_1$. By using the fact that the noncommutative  Hardy-Littlewood maximal operator is of weak type (1,1) (see \eg \cite{Mei07}), it is not difficult to see that the maximal operator of $\tilde{M}_j$ is of weak type (1,1). Thus we can find a projection $e_{1}\in\A_\pi$ such that
\Bes
\sup_{j\in\Z} \big\|
e_{1} \tilde{M}_{j}f e_{1}\big\|_{L_\infty(\A)} \leq
\lam \qquad \mbox{and} \qquad \lambda\varphi \big( \
1_\A -
e_{1} \big)\lc\|\Om\|_1\|f\|_{L_1(\A)}.
\Ees

Next we utilize Theorem \ref{t:19sq} to construct other projection. By the definition of infimum,  there exists a decomposition  $T_{j}f=g_{j}+h_{j}$ satisfying $$\|\{g_{j}\}\|_{L_{1,\infty}(\A; \ell_{2}^{c})}+\|\{h_{j}\}\|_{L_{1,\infty}(\A; \ell_{2}^{r})}\leq\fr{1}{2}\C_\Om\|f\|_{L_1(\A)}.$$
We now take $e_{2}=\chi_{(0,\lambda]}\big((\sum\limits_{j\in\Z}|g_{j}|^{2})^{\frac{1}{2}}\big)$ and $e_{3}=\chi_{(0,\lambda]}\big((\sum\limits_{j\in\Z}|h^{\ast}_{j}|^{2})^{\frac{1}{2}}\big)$, then
$$
\big\|\big((\sum_{j\in\Z}|g_{j}|^{2})^{\frac{1}{2}}\big)
e_{2}\big\|_{L_\infty(\A)} \leq
\lambda \qquad \text{and} \qquad \lambda\varphi \big( 1_\A -
e_{2} \big) \lesssim\C_\Om\|f\|_{L_1(\A)}.
$$
Also for $e_{3}$, we have
$$
\big\|\big((\sum_{j\in\Z}|h^{\ast}_{j}|^{2})^{\frac{1}{2}}\big)
e_{3}\big\|_{L_\infty(\A)} \leq
\lambda \qquad \text{and} \qquad \lambda\varphi \big( 1_\A -
e_{3} \big) \lesssim\C_\Om\|f\|_{L_1(\A)}.
$$

Let $e=e_1\bigwedge e_{2}\bigwedge e_{3}$. Then it is easy to see that
$$\sup_{j\in\Z} \big\|
e \tilde{M}_{j}f e\big\|_{L_\infty(\A)} \leq\lam,\quad
\lambda\varphi \big( 1_\A -
e \big) \lesssim\C_\Om\|f\|_{L_1(\A)}.$$
Hence to finish the proof of Theorem \ref{t:19dya}, it is sufficient to show
$$\sup_{j\in\Z} \big\|eT_jfe\big\|_{L_\infty(\A)}\lc\lambda.$$
Recall the definition of $L_\infty(\A)$, $\|f\|_{L_\infty(\A)}=\|f\|_{\A}$. Then we get
\Bes
\big\|e T_jfe\big\|_{L_\infty(\A)}\leq \big\|eg_{j}e\big\|_{\A}+\big\|eh_{j}e\big\|_{\A}
=\big\|eg_{j}e\big\|_{\A}+\big\|eh^{\ast}_{j}e\big\|_{\A}.
\Ees
Now using polar decomposition $g_j=u_j|g_j|$ and $h_j^*=v_j|h_j^*|$, we continue to estimate the above term as follows,
\begin{align*}
\big\|eu_{j}|g_{j}|e\big\|_{\A}+\big\|ev_{j}|h^{\ast}_{j}|e\big\|_{\A}
&\ \leq\big\||g_{j}|e\big\|_{\A}+\big\||h^{\ast}_{j}|e\big\|_{\A}
=\big\|e|g_{j}|^{2}e\big\|^{\frac{1}{2}}_{\A}+\big\|e|h^{\ast}_{j}|^{2}e\big\|^{\frac{1}{2}}_{\A}\\
   &\ \leq \big\|e\sum_{j\in\Z}|g_{j}|^{2}e\big\|^{\frac{1}{2}}_{\A}+\big\|e\sum_{j\in\Z}|h^{\ast}_{j}|^{2}e\big\|^{\frac{1}{2}}_{\A}\\
   &\ = \big\|(\sum_{j\in\Z}|g_{j}|^{2})^{\frac{1}{2}}e_{2}e\big\|_{\A}+\big\|(\sum_{j\in\Z}|h^{\ast}_{j}|^{2})^{\frac{1}{2}}e_{3}e\big\|_{\A}\lc\lambda.
\end{align*}
Hence we finish the proof of Theorem \ref{t:19dya}.

\emph{Step 4}. We reduce the study of square function to a linear operator. To simplify the notation, we still use $\Om$ in \eqref{e:19Tj}, i.e.
\Be\label{e:19Tjn}
T_jf(x)=\int_{\R^d}K_j(x-y){{\Om}(x-y)}f(y)dy,\quad \text{with}\quad K_j(x)=\phi_j(x)|x|^{-d},\Ee
but we suppose that $\Om$ satisfies the cancelation property $\int_{\S^{d-1}}\Om(\tet)d\si(\tet)=0$. To linearize the square function, we use the following noncommutative Khintchine's inequality in $L_{1,\infty}(\A,\ell_2^{rc})$ which was recently established by L. Cadilhac \cite{Cad19}.
\begin{lemma}
Let $\{\eps_j\}_j$ be a Rademacher sequence on a probability space $(\mathfrak{m},P)$. Suppose that $f=\{f_j\}_j$ is a finite sequence in $L_{1,\infty}(\A)$. Then we have
\Bes
\Big\|\sum_{j\in\Z}f_j\eps_j\Big\|_{L_{1,\infty}(L_{\infty}(\mathfrak{m})\overline{\otimes}\A)}
\approx\|\{f_j\}_j\|_{L_{1,\infty}(\A;\ell_2^{rc})}.
\Ees
\end{lemma}

Now by the preceding lemma, Theorem \ref{t:19sq} immediately follows from the following result.

\begin{theorem}\label{p:19m}
Suppose that $\Om$ satisfies \eqref{e:19Hom}, $\Om\in L(\log^+ L)^{2}(\S^{d-1})$  and the cancelation property $\int_{\S^{d-1}}\Om(\tet)d\si(\tet)=0$.
Let $T_j$ be defined in \eqref{e:19Tjn}. Assume $\{\eps_j\}_j$ is the Rademacher sequence on a probability space $(\mathfrak{m},P)$.
Define $Tf(x,z)=\sum_jT_jf(x)\eps_j(z)$ and set the tensor trace $\tilde{\vphi}=\int_{\mathfrak{m}}\otimes\vphi$. Then $T$ maps $L_1(\A)$ to $L_{1,\infty}(L_{\infty}(\mathfrak{m})\overline{\otimes}\A)$, i.e. for any $\lam>0, f\in L_1(\A)$,
\Bes
\lam\tilde{\vphi}\{|T f|>\lam\}\lc\C_\Om\|f\|_{L_{1}(\A)},
\Ees
where the constant $\C_\Om$ only depends on $\Om$ and the dimension.
\end{theorem}
At present our main result Theorem \ref{t:19} is reduced to Theorem \ref{p:19m}. In the rest of this paper, we give effort to the proof of Theorem \ref{p:19m}.
\vskip0.24cm

\section {Proof of Theorem \ref{p:19m}}\label{s:194}
\vskip0.24cm
In this section we give the proof of Theorem \ref{p:19m} based on some lemmas, their proofs will be given in Section \ref{s:195} and Section \ref{s:196}. We first introduce the noncommutative Calder\'on-Zygmund decomposition.

\subsection{Noncommutative Calder\'on-Zygmund decomposition}\label{s:1941}\quad
\vskip0.24cm

By the standard density argument, we only need to consider the following dense class of $L_1(\A)$
$$
\A_{c,+}=\{f:\R^d\rta\M\ | f\in\A_+,\ \text{$\overrightarrow{\rm {supp}}\,f$ is compact}\}.
$$
Here $\overrightarrow{\rm {supp}}\,f$ represents the support of $f$ as an operator-valued function in $\R^d$, which means that $\overrightarrow{\rm {supp}}\ f=\{x\in\R^d: \|f(x)\|_{\M}\neq0\}$. Let $\Om\in L(\log^+L)^2(\S^{d-1})$. Set a constant
\Be\label{e:19constantom}
\mathcal{C}_\Om=\|\Om\|_{L(\log^+L)^2}
+\int_{\S^{d-1}}|\Om(\tet)|\big(1+[\log^+({|\Om(\tet)|}/{\|\Om\|_{1}})]^2\big)d\si(\tet),
\Ee
where $\log^+a=0$ if $0<a<1$ and $\log^+a=\log a$ if $a\geq1$. Since $\|\Om\|_{L(\log^+L)^2}<+\infty$, one can easily check that $\mathcal{C}_\Om$ is a finite constant.
Now we fix $f\in \A_{c,+}$, set $f_k=\mathsf{E}_kf$ for all $k\in\Z$. Then the sequence $\{f_k\}_{k\in\Z}$ is a positive dyadic martingale in $L_1(\A)$. Applying the so-called Cuculescu construction introduced in \cite[Lemma 3.1]{Par09} at level ${\lam}{\C_\Om}^{-1}$, we get the following result.
\begin{lemma}\label{l:19cucu}
There exists a decreasing sequence $\{q_k\}_{k\in\Z}$ depending on $f$ and ${\lam}{\C_\Om}^{-1}$,
where $q_k$ is a projection in $\A_\pi$ satisfying the following conditions:
\begin{enumerate}[\rm (i).]
\item $q_k$ commutes with $q_{k-1}f_kq_{k-1}$ for every $k\in\Z$;
\item $q_k$ belongs to $\A_k$ for every $k\in\Z$ and $q_kf_kq_k\leq{\lam}{\C_\Om}^{-1}q_k$;
\item Set $q=\bigwedge_{k\in\Z}q_k$. We have the following inequality
\Bes
\vphi(1_\A-q)\leq\lam^{-1}\C_\Om\|f\|_{L_1(\A)};
\Ees
\item The expression of $q_k$ can be written  as follows: for some negative integer $m\in\Z$
\Bes
q_k=\begin{cases}
1_\A\ \ &\qquad\text{if\ $k<m$,} \\
\chi_{(0,\lam\C_\Om^{-1}]}(f_k)\ &\qquad\text{if\ $k=m$},\\
\chi_{(0,\lam\C_\Om^{-1}]}(q_{k-1}f_kq_{k-1}) &\qquad\text{if\ $k>m$}.
\end{cases}
\Ees
\end{enumerate}
\end{lemma}

Below we introduce another expression of the projection $q_k$ given in the previous lemma as done in \cite{Par09}. We point out that such kind of expression will be quite helpful when we give some estimates to the terms related to $q_k$. In fact we can write
$
q_k=\sum_{Q\in\mathcal{Q}_k}\xi_Q\chi_{Q}
$
for all $k\in\Z$, where $\xi_Q$ is a projection in $\M$ which satisfies the following conditions:
\begin{enumerate}[\rm (i).]
\item $\xi_Q$ has the following explicit expression: $\widehat{Q}$ below is the father dyadic cube of $Q$,
\Bes
\xi_Q=\begin{cases}
1_\M\ \ &\qquad\text{if $k<m$,} \\
\chi_{(0,\lam\C_\Om^{-1}]}(f_Q)\ &\qquad\text{if $k=m$},\\
\chi_{(0,\lam\C_\Om^{-1}]}(\xi_{\widehat{Q}}f_Q\xi_{\widehat{Q}}) &\qquad\text{if $k>m$}.
\end{cases}
\Ees
\item $\xi_Q\in\M_{\pi}$ and $\xi_Q\leq\xi_{\widehat{Q}}$;
\item $\xi_Q$ commutes with $\xi_{\widehat{Q}}f_Q\xi_{\widehat{Q}}$ and  $\xi_Qf_Q\xi_Q\leq \C^{-1}_\Om\lam\xi_Q$.
\end{enumerate}

Define the projection $p_k=q_{k-1}-q_k$. By applying the above more explicit expression, we see that $p_k$ equals to
$
\sum_{Q\in\Q_k}(\xi_{\widehat{Q}}-\xi_Q)\chi_Q=:\sum_{Q\in\Q_k}{\pi_Q\chi_Q}
$
where $\pi_Q=\xi_{\widehat{Q}}-\xi_Q$.
Then it is easy to see that all $p_k$s are pairwise disjoint and
$\sum_{k\in\Z}p_k=1_\A-q$.

Now we define the associated good functions and bad functions related to $f$ as follows:
\Bes
f=g+b,\quad g=\sum_{i,j\in\widehat{\Z}}p_if_{i\vee j}p_j,\quad b=\sum_{i,j\in\widehat{\Z}}p_i(f-f_{i\vee j})p_j
\Ees
where we set $p_\infty=q$,  $\widehat{\Z}=\Z\cup\{\infty\}$  and $i\vee j=\max (i,j)$. If $i$ or $j$ is infinite, $i\vee j$ is just $\infty$ and $f_\infty=f$ by definition. We further decompose $g$  as the diagonal terms and the off-diagonal terms
\Bes
g_d=qfq+\sum_{k\in\Z}p_kf_kp_k,\quad g_{\text{off}}=\sum_{i\neq j}p_if_{i\vee j}p_j+qf(1_\A-q)+(1_\A-q)fq.
\Ees
The proofs for diagonal terms $g_d$ and off-diagonal terms $g_\off$ will be different as we shall see below. For the bad function $b$, we can deal with the diagonal and off-diagonal terms uniformly. So it is unnecessary for us to  decompose it as that for the good functions.
By the linearity of  $T$, we get that
\Bes
\tilde{\vphi}(|Tf|>\lam)\leq\tilde{\vphi}(|Tg|>\lam/2)+\tilde{\vphi}(|Tb|>\lam/2).
\Ees

In the following we give estimates for the good and bad functions, respectively. Before that we state a lemma to construct a projection in $\A$ such that the proof can be reduced to the case that the operators are restricted on this projection.
\begin{lemma}\label{l:19excep}
There exists a projection $\zet\in\A_\pi$ which satisfies the following conditions
\begin{enumerate}[\rm (i).]
\item $\lam\vphi(1_\A-\zet)\lc\C_\Om\|f\|_{L_1(\A)}$.
\item If $Q_0\in\Q$ and $x\in (2^{101}+1)Q_0$, then $\zet(x)\leq1_\M-\xi_{\widehat{Q_0}}+\xi_{Q_0}$ and $\zet(x)\leq\xi_{Q_0}$.
\end{enumerate}
\end{lemma}

The proof of this lemma can be easily modified from that of \cite[Lemma 4.2]{Par09}. Here the exact value of  $2^{101}+1$ above is not essential and the reason we choose this value is just for convenience in later calculation (see \eqref{e:19exec} later). Now let us  consider the bad functions first since our method presented here is also needed for the good functions.
\vskip0.24cm
\subsection{Estimates for the bad functions}\label{s:1942}\quad
\vskip0.24cm

We first use Lemma \ref{l:19excep} to reduce the study of the operator $T$ to that of $\zet T\zet$. Split $Tb$ into four terms as follows
\Bes
(1_\A-\zeta)Tb(1_\A-\zet)+\zet Tb(1-\zet)+(1-\zet)Tb\zet+\zet Tb\zet.
\Ees
By the property (i) in Lemma \ref{l:19excep}, we get that
\Bes
\tilde{\vphi}(|Tb|>\lam/2)\lc\vphi(1_\A-\zet)+\tilde{\vphi}(|\zet Tb\zet|>\lam/4)
\lc\lam^{-1}\C_\Om\|f\|_{L_1(\A)}+\tilde{\vphi}(|\zet Tb\zet|>\lam/4).
\Ees
Therefore it is enough to show that the term $\tilde{\vphi}(|\zet Tb\zet|>\lam/4)$ satisfies our desired estimate. Recall the bad function
\Bes
b=\sum_{k\in\Z}p_k(f-f_k)p_k+\sum_{s\geq1}\sum_{k\in\Z}p_k(f-f_{k+s})p_{k+s}+p_{k+s}(f-f_{k+s})p_k
=:\sum_{s=0}\sum_{k\in\Z}b_{k,s},
\Ees
where
\Be\label{e:19db}
b_{k,0}=p_k(f-f_k)p_k,\quad b_{k,s}=p_k(f-f_{k+s})p_{k+s}+p_{k+s}(f-f_{k+s})p_k.
\Ee
By the definition of $T$, we further rewrite $Tb$ as follows: for any $x\in\R^d$ and $z\in\mathfrak{m}$,
\[Tb(x,z)=\sum_{j\in\mathbb{Z}}T_j[\sum_{s\geq0}\sum_{n\in\mathbb{Z}}b_{n-j,s}](x)
\eps_j(z)=\sum_{s\geq0}\sum_{n\in\mathbb{Z}}\sum_{j\in\mathbb{Z}}T_jb_{n-j,s}(x)\eps_j(z).\]

For any $Q\in\Q_{n-j+s}$, set $Q_{n-j}\in\Q_{n-j}$ as the $s$th ancestor of $Q$. Consider $x$ in the support of $\zet$ (i.e. $\zet(x)\neq0$) and let $n<100$, then we get that for all $s$, $\zet(x)T_jb_{n-j,s}(x)\zet(x)$ equals to
\Be\label{e:19exec}
\begin{split}
&\sum_{Q\in\Q_{n-j+s}}\zet(x)\int_Q K_j(x-y)b_{n-j,s}(y)dy\zet(x)\\
&=\sum_{Q\in\Q_{n-j+s}}\zet(x)\chi_{((2^{101}+1)Q_{n-j})^c}(x)\int_Q K_j(x-y)\\
&\quad\times\big[\pi_{Q_{n-j}}(f(y)-f_Q)\pi_Q
+\pi_Q(f(y)-f_Q)\pi_{Q_{n-j}}\big]dy\zet(x)\\
&=0
\end{split}
\Ee
where in the first equality we apply $\zet(x)\pi_{Q_{n-j}}=0$ if $x\in (2^{101}+1)Q_{n-j}$ by the property (ii) of $\zet$ in Lemma \ref{l:19excep} and the second inequality follows from the fact $x\in ((2^{101}+1)Q_{n-j})^c$ and $y\in Q$ implies that $|x-y|\geq2^{100+j-n}$ which is a contradiction with the support of $K_j$ and $n<100$.
Therefore we get
$$\zet Tb \zet=\zet\sum_{n\geq100}\sum_{s\geq0}\sum_{j\in
\mathbb{Z}}T_jb_{n-j,s}\eps_j\zet.
$$
Hence, to finish the proof related to the bad functions, it suffices to verify the following estimate:
\begin{equation}\label{e:19boff}
\tilde{\vphi}\big(\big|\zet\sum_{n\geq100}\sum_{s\geq0}\sum_{j\in
\mathbb{Z}}T_jb_{n-j,s}\eps_j\zet\big|>\lambda/4\big)\lc {\lambda}^{-1}\C_\Om\|f\|_{L_1(\A)}.
\end{equation}

Some important decompositions play key roles in the proof of \eqref{e:19boff}. We present them by some lemmas, which will be proved in Section \ref{s:195}. It should be pointed out here that the methods used here also work for the good functions, which will be clear in the next subsection.

The first lemma shows that,  \eqref{e:19boff} holds if $\Om$ is restricted in some subset of $\S^{d-1}$. More precisely, for fixed $n\ge100$ and $s\geq 0$, denote $D^\iota=\{\theta\in\S^{d-1}:\,|\Om(\theta)|\geq2^{\iota (n+s)}\|\Om\|_{1}\}$, where $\iota>0$ will be chosen later.
Let $T_{j,\iota}^{n,s}$ be defined by
\Be\label{e:19tjiota}
T_{j,\iota}^{n,s}h(x)=\int_{\R^d}\Om\chi_{D^{\iota}}\Big(\fr{x-y}{|x-y|}\Big)
K_j(x-y)\cdot h(y)dy.
\Ee

\begin{lemma}\label{l:19cur}
Suppose $\Om\in L(\log^+ L)^2(\S^{d-1})$. With all notation above, we get that
\Bes
\tilde{\vphi}\big(\big|\zet\sum_{n\geq100}\sum_{s\geq0}\sum_{j\in
\mathbb{Z}}T_{j,\iota}^{n,s}b_{n-j,s}\eps_j\zet\big|>\lambda/8\big)\lc {\lambda}^{-1}\C_\Om\|f\|_{L_1(\A)}.
\Ees
\end{lemma}

Thus, by Lemma \ref{l:19cur}, to finish the proof for bad functions, it suffices to verify (\ref{e:19boff}) under the condition that for fixed $n\geq100$ and $s\geq0$ the kernel function $\Omega$ satisfies $\|\Om\|_{\infty}\leq2^{\iota (n+s)}\|\Om\|_{1}$
in each $T_j$.

In the following, we introduce the
\emph{microlocal decomposition}\,of kernel. To do this, we give a partition of unity on the unit surface $\S^{d-1}$.
Let $k\geq100$. Choose $\{e^k_v\}_{v\in\Theta_k}$ be a collection of unit vectors on $\S^{d-1}$
which satisfies the following two conditions:

(a)\ $|e^k_v-e^k_{v'}|\geq 2^{-k\ga-4}$ if $v\neq v'$;

(b)\ If $\theta\in \S^{d-1}$, there exists an $e^k_v$ such that $|e^k_v-\theta|\leq 2^{-k\ga-3}$.

\noindent The constant $0<\ga<1$ in (a) and (b) will be chosen later. To choose such an $\{e_v^k\}_{v\in\Theta_k}$, we simply take a maximal collection $\{e^k_v\}_v$ for which (a) holds and then (b) holds automatically by the maximality. Notice that there are $C2^{k\ga(d-1)}$
elements in the collection $\{e^k_v\}_v$. For every $\tet\in\S^{d-1}$,
there only exists finite $e^k_v$ such that $|e^k_v-\tet|\leq2^{-k\ga-4}$. Now we can construct an associated partition of unity
on the unit surface $\S^{d-1}$. Let $\eta$ be a smooth, nonnegative, radial function with $\eta(u)=1$ for
$|u|\leq \fr{1}{2}$ and $\eta(u)=0$ for $|u|>1$. Define $$\tilde{\Ga}^k_v(u)=\eta\Big(2^{k\ga}(\fr{u}{|u|}-e^k_v)\Big),\quad
\Ga^k_v(u)=\tilde{\Ga}^k_v(u)\Big(\sum\limits_{v\in\Theta_k}\tilde{\Ga}^k_v(u)\Big)^{-1}. $$
Then it is easy to see that $\Ga^k_v$ is homogeneous of degree $0$ with $\sum\limits_{v\in\Theta_{k}}\Ga^k_v(u)=1$, for all $u\neq0$ and all $k$.
Now we define operator $T_j^{n,s,v}$ by
\Be\label{e:19Tjnsv}
T_j^{n,s,v}h(x)=\int_{\R^d}\Om(x-y)\Ga^{n+s}_v(x-y)\cdot K_j(x-y)\cdot h(y)dy.
\Ee
Then it is easy to see that $T_j=\sum\limits_{v\in\Theta_{n+s}}T_j^{n,s,v}.$

In the sequel, we will use the Fourier transform since we need to separate the phase in frequency space into different directions. Hence we define a Fourier multiplier operator by
$$\widehat{G_{k,v}h}(\xi)=\Phi(2^{k\ga}\inn{e^k_v}{{\xi}/{|\xi|}})\hat{h}(\xi),$$
where $\hat{h}$ is the Fourier transform of $h$ and $\Phi$ is a smooth, nonnegative, radial function such that
$0\leq\Phi(x)\leq1$ and $\Phi(x)=1$ on $|x|\leq2$, $\Phi(x)=0$ on $|x|>4$.
Now we can split $T_j^{n,s,v}$ into two parts:
$
T_j^{n,s,v}=G_{n+s,v}T_j^{n,s,v}+(I-G_{n+s,v})T_j^{n,s,v}.
$

The following lemma gives the $L^2$ estimate involving $G_{n+s,v}T_j^{n,s,v}$, which will be proved in Section \ref{s:195}.

\begin{lemma}{\label{l:19L^2}}
Let $n\geq100$ and $s\geq0$. Suppose $\|\Om\|_{\infty}\leq2^{\iota (n+s)}\|\Om\|_{1}$ in each $T_j$. With all notation above,  we get the following estimate
$$\sum\limits_j\Big\|\sum\limits_{v\in\Theta_{n+s}}G_{n+s,v}T_j^{n,s,v}b_{n-j,s}\Big\|^2_{L_2(\A)}
\lc2^{-(n+s)\ga+2(n+s)\iota}\lam\C_\Om\|f\|_{L_1(\A)}.$$
\end{lemma}

The terms involving $(I-G_{n+s,v})T_j^{n,s,v}$ are more complicated. For convenience, we set $L^{n,s,v}_j=(I-G_{n+s,v})T_j^{n,s,v}$.
In Section \ref{s:195}, we shall prove the following lemma.
\begin{lemma}\label{l:19L^1}
Let $n\geq100$ and $s\geq0$. Suppose $\|\Om\|_{\infty}\leq2^{\iota (n+s)}\|\Om\|_{1}$ in each $T_j$. With all notation above, then there exists
 a positive constant $\alp$ such that
$$\sum\limits_j\sum_{v\in\Theta_{n+s}}\big\|L_j^{n,s,v}b_{n-j,s}\big\|_{L_1(\A)}\lc2^{-(n+s)\alp}\C_\Om\|f\|_{L_1(\A)}.$$
\end{lemma}

We now complete the proof of \eqref{e:19boff}. It is sufficient to prove \eqref{e:19boff} under the condition that for all fixed $n\geq100$ and $s\geq0$, $\|\Om\|_{\infty}\leq2^{\iota (n+s)}\|\Om\|_{1}$ in $T_j$.  By Chebyshev's inequality and triangle inequality, we get
\begin{equation*}
\begin{split}
&\quad \tilde{\vphi}\big(\big|\zet\sum_{n\geq100}\sum_{s\geq0}\sum_{j\in
\mathbb{Z}}T_jb_{n-j,s}\eps_j\zet\big|>\lambda/8\big)\\
&\lc {\lam^{-2}}\Big\|\zet\sum\limits_{n\geq100}\sum_{s\geq0}\sum\limits_{j}\sum\limits_{v\in\Theta_{n+s}}G_{n+s,v}T_j^{n,s,v}b_{n-j,s}\eps_j\zet\Big\|_{L_2(L_\infty\mathfrak{m}\ot\A)}^2
\\
&\quad +\lam^{-1}\sum_{n\geq100}\sum_{s\geq0}\sum\limits_j\sum_{v\in\Theta_{n+s}}\big\|\zet L_j^{n,s,v}b_{n-j,s}\eps_j\zet\big\|_{L_1(L_\infty(\mathfrak{m})\ot\A)}\\
&=:I+II.\\
\end{split}
\end{equation*}

First we consider the term $I$. Recall that  $\{\eps_j\}_j$ is a Rademacher sequence on a probability space $(\mathfrak{m},P)$. So we have the following orthogonal equality
\Be\label{e:19orth}
\Big\|\sum_{j\in\Z}\eps_ja_j\Big\|^2_{L_2(L_\infty(\mathfrak{m})\ot\A)}=\sum_j\|a_j\|_{L_2(\A)}^2.
\Ee
Choose $0<\iota<\fr{\ga}{2}<\fr12$. By triangle inequality, the above orthogonal equality, using H\"older's inequality to remove $\zet$ since $\zet$ is a projection in $\A$, and finally by Lemma \ref{l:19L^2}, we get that
\begin{equation*}
\begin{split}
I&\lc\lam^{-2}\Big(\sum\limits_{n\geq100}\sum_{s\geq 0}\Big\|\sum\limits_{j}\eps_j\zet
\sum\limits_{v\in\Theta_{n+s}}G_{n+s,v}T_j^{n,s,v}b_{n-j,s}\zet\Big\|_{L_2(L_\infty(\mathfrak{m})\ot\A)}\Big)^2\\
&\lc\lam^{-2}\Big(\sum\limits_{n\geq100}\sum_{s\geq 0}\Big(\sum\limits_{j}\Big\|\zet\sum\limits_{v\in\Theta_{n+s}}G_{n+s,v}T_j^{n,s,v}b_{n-j,s}\zet\Big\|^2_{L_2(\A)}\Big)^{1/2}\Big)^2\\
&\lc\lam^{-2}\Big(\sum\limits_{n\geq100}\sum_{s\geq 0}\big(2^{-(n+s)\ga+2(n+s)\iota}\C_\Om\lam\|f\|_{L_1(\A)})^{\fr{1}{2}}\Big)^{2}
\lc \lam^{-1}\C_\Om\|f\|_{L_1(\A)}.
\end{split}
\end{equation*}

For the term $II$, by the fact that $\{\eps_j\}_{j\in\Z}$ is a bounded sequence,  using H\"older's inequality to remove $\zet$ and by Lemma \ref{l:19L^1}, we can get that
\Bes
\begin{split}
II&\lc\lam^{-1}\sum_{n\geq100}\sum_{s\geq0}\sum\limits_j\sum_{v\in\Theta_{n+s}}\big\|L_j^{n,s,v}b_{n-j,s}\big\|_{L_1(\A)}\\
&\lc\lam^{-1}\sum_{n\geq100}\sum_{s\geq0}2^{-(n+s)\alp}\C_\Om\|f\|_{L_1(\A)}\lc\C_\Om\lam^{-1}\|f\|_{L_1(\A)}.
\end{split}
\Ees

Hence we complete the proof of \eqref{e:19boff} based on Lemmas \ref{l:19cur}, \ref{l:19L^2} \ref{l:19L^1}. All their proofs will  be given in Section \ref{s:195}.

\vskip0.24cm

\subsection{Estimates for the good functions}\label{s:1943}\quad
\vskip0.24cm

Now we turn to the estimates for good functions. The proofs of diagonal terms and off-diagonal terms will be quite different. We first consider the diagonal terms which is simpler since they behave similar to that in the classical Calder\'on-Zygmund decomposition.
Following the classical strategy, we should first establish the $L_2$ boundedness of $T$. In this situation, the condition for the kernel $\Om$ in fact can be relaxed to $\Om\in L(\log^+L)^{1/2}(\S^{d-1})$.
\begin{lemma}\label{l:19tgl2}
Suppose that $\Om$ satisfy \eqref{e:19Hom}, $\Om\in L(\log^+ L)^{1/2}(\S^{d-1})$ and the cancelation property $\int_{\S^{d-1}}\Om(\tet)d\si(\tet)=0$. Then we have
\Bes
\|Tf\|_{L_2(L_\infty(\mathfrak{m})\ot\A)}\lc\|f\|_{L_2(\A)},
\Ees
where the implicit constant above depends only on the dimension and $\Om$.
\end{lemma}
\begin{remark*}
It should be pointed out that the cancelation condition $\int_{\S^{d-1}}\Om(\tet)d\tet=0$ in Theorem \ref{p:19m} is only used in this lemma to guarantee the $L_2$ boundedness of $T$.
\end{remark*}
The proof of Lemma \ref{l:19tgl2} will be given in Section \ref{s:196}. Based on this lemma, we could prove required bound for the diagonal term $g_d$ of good functions as follows. By using the property of $q_k$s in Lemma \ref{l:19cucu}, J. Parcet \cite{Par09} obtained the following basic property of $g_d$:
\Be\label{e:19gdes}
\|g_d\|_{L_1(\A)}\lc\|f\|_{L_1(\A)},\quad \|g_d\|_{L_\infty(\A)}\lc\lam\C_\Om^{-1}.
\Ee
By Lemma \ref{l:19tgl2}, it is not difficult to see that the $L_2$ norm of $T$ is bounded by $\C_\Om$ (see the details in Subsection \ref{s:1961} for its proof). Therefore we get the estimate for $g_d$ as follows

\Bes
\begin{split}
\tilde{\vphi}(|Tg_d|>\lam/4)&\lc\lam^{-2}\|Tg_d\|^2_{L_2(L_\infty(\mathfrak{m})\ot\A)}
\lc\lam^{-2}\C_\Om^2\|g_d\|^2_{L_2(\A)}\\
&\lc\lam^{-2}\C_\Om^2\|g_d\|_{L_1(\A)}\|g_d\|_{L_\infty(\A)}\lc\lam^{-1}\C_\Om\|f\|_{L_1(\A)}.
\end{split}
\Ees
where in the first inequality we use Chebyshev's inequality, the second inequality follows from Lemma \ref{l:19tgl2}, the third and fourth inequalities just follow from \eqref{e:19gdes}.

In the remaining parts of this subsection, we give effort to the estimate of $g_\off$. We first use Lemma \ref{l:19excep} to reduce the proof to the case $\zet Tg_\off\zet$.
In fact
\Bes
Tg_\off=(1_\A-\zeta)Tg_\off(1_\A-\zet)+\zet Tg_\off(1-\zet)+(1-\zet)Tg_\off\zet+\zet Tg_\off\zet.
\Ees
By Lemma \ref{l:19excep} and the same argument as done for the bad functions, it is sufficient to consider the last term $\zet Tg_\off\zet$ above. Thus our goal is to prove that
\Be\label{e:19goffg}
\tilde\vphi(|\zet Tg_\off\zet|>\lam/8)\lc\lam^{-1}\C_\Om\|f\|_{L_1(\A)}.
\Ee

Next we introduce another expression of the off-diagonal terms $g_\off$ and related estimates which were proved by J. Parcet in \cite{Par09}.
\begin{lemma}\label{l:19gbasic}
Let $df_s$ be martingale difference. We could rewrite $g_{\off}$ as follows
\Bes
g_{\off}=\sum_{s\geq1}\sum_{k\in\Z}p_kdf_{k+s}q_{k+s-1}+q_{k+s-1}df_{k+s}p_k=:\sum_{s\geq1}\sum_{k\in\Z}g_{k,s}=:\sum_{s\geq1} g_{(s)}.
\Ees
The martingale difference sequence of $g_{(s)}$ satisfies ${dg_{(s)}}_{k+s}=g_{k,s}$ and
$\supp^* g_{k,s}\leq p_k\leq1_\A-q_k$,
where $\supp^*$ is weak support projection defined by $\supp^*a=1_\A-\mathbf{q}$, $\mathbf{q}$ is the greatest projection satisfying $\mathbf{q}a\mathbf{q}=0$. Meanwhile, we have the following estimates
\Bes
\sup_{s\geq1}\|g_{(s)}\|_{L_2(\A)}^2=\sup_{s\geq1}\sum_{k\in\Z}\|g_{k,s}\|_{L_2(\A)}^2\lc\lam\C_\Om^{-1}\|f\|_{L_1(\A)}.
\Ees
\end{lemma}

The strategy to deal with the off-diagonal terms $g_\off$ is similar to that we use in the proof for the bad functions, although the technical proofs may be different. By the expression of $g_\off$ in Lemma \ref{l:19gbasic} and the formula $f=\sum_{n\in\Z}df_n$, we could write
\Bes
\begin{split}
\zet Tg_\off\zet&=\zet\sum_{s\geq1}\sum_{j\in\Z} \eps_jT_jg_{(s)}\zet=\zet\sum_{s\geq1}\sum_{j\in\Z}\sum_{n\in\Z} \eps_jT_jd\big(g_{(s)}\big)_{n-j+s}\zet\\
&=\zet\sum_{s\geq1}\sum_{j\in\Z}\sum_{n\in\Z} \eps_jT_jg_{n-j,s}\zet
=\zet\sum_{s\geq1}\sum_{n\geq100}\sum_{j\in\Z} \eps_jT_jg_{n-j,s}\zet
\end{split}
\Ees
where the last equality follows from the fact: if $\zet(x)\neq0$ and $n<100$, we get that $T_jg_{n-j,s}(x)=0$ for all $s\geq1$ by the property in (ii) of Lemma \ref{l:19excep}, $\supp^* g_{k,s}\leq p_k\leq1_\A-q_k$ in Lemma \ref{l:19gbasic} and the similar arguments in \eqref{e:19exec}.

By Chebyshev's inequality, triangle inequality, $\zet$ is a projection in $\A$ and the orthogonal equality \eqref{e:19orth}, we then get that
\Bes
\begin{split}
\tilde{\vphi}(|\zet Tg_\off\zet|>\lam)&\lc\lam^{-2}\|\zet Tg_\off\zet\|_{L_2(L_\infty(\mathfrak{m})\ot\A)}^2\\
&\lc\lam^{-2}\bigg(\sum_{s\geq1}\sum_{n\geq100}\Big(\sum_j\|T_j g_{n-j,s}\|^2_{L_2(\A)}\Big)^{1/2}\bigg)^2.
\end{split}
\Ees
Hence to finish the proof for the off-diagonal terms $g_\off$, it is sufficient to show that
\Be\label{e:19boffgoal}
\sum_{s\geq1}\sum_{n\geq100}\Big(\sum_j\|T_jg_{n-j,s}\|^2_{L_2(\A)}\Big)^{1/2}
\lc(\C_\Om\lam\|f\|_{L_1(\A)})^{1/2}
\Ee

As done in the proof for  the bad functions, we first show that \eqref{e:19boffgoal} holds if $\Om$ is restricted in $D^\iota=\{\tet\in\S^{d-1}:|\Om(\tet)|\geq2^{\iota(n+s)}\|\Om\|_1\}$, where $\iota\in(0,1)$.  Recall the definition of $T_{j,\iota}^{n,s}$ in \eqref{e:19tjiota}. Then  we have the following lemma.

\begin{lemma}\label{l:19gomiota}
Suppose $\Om\in L(\log^+ L)^2(\S^{d-1})$. With all notation above, we get that
\Bes
\sum_{s\geq1}\sum_{n\geq100}\Big(\sum_j\|T_{j,\iota}^{n,s}g_{n-j,s}\|^2_{L_2(\A)}\Big)^{1/2}
\lc(\C_\Om\lam\|f\|_{L_1(\A)})^{1/2}.
\Ees

\end{lemma}

The proof of Lemma \ref{l:19gomiota} will be given in Section \ref{s:196}. By Lemma \ref{l:19gomiota}, to prove \eqref{e:19boffgoal},  we only need to show \eqref{e:19boffgoal} under the condition that the kernel function $\Omega$ satisfies $\|\Om\|_{\infty}\leq2^{\iota (n+s)}\|\Om\|_{1}$ in each $T_j$.
For each fixed $s\geq1$ and $n\geq100$, we make a microlocal  decomposition of $T_j$ as follows:
\Bes
T_j=\sum_{v\in\Theta_{n+s}}T_j^{n,s,v},\quad T_j^{n,s,v}=G_{n+s,v}T_j^{n,s,v}+(I-G_{n+s,v})T_j^{n,s,v}.
\Ees
Here all notation $T_j^{n,s,v}$, $G_{n+s,v}$ are the same as those in the proof of the bad functions.
\begin{lemma}\label{l:19gl2}
Let $n\geq100$ and $s\geq1$. Suppose that $\|\Om\|_{\infty}\leq2^{\iota (n+s)}\|\Om\|_{1}$ in each $T_j$. Then we get the following estimate
\Bes
\Big\|\sum_{v\in\Th_{n+s}}G_{n+s,v}T_j^{n,s,v}g_{n-j,s}\Big\|^2_{L_2(\A)}\lc2^{-(n+s)(\ga-2\iota)}\|\Om\|^2_1\|g_{n-j,s}\|_{L_2(\A)}^2.
\Ees
\end{lemma}

\begin{lemma}\label{l:19gl1}
Let $n\geq100$ and $s\geq1$. Suppose $\|\Om\|_{\infty}\leq2^{\iota (n+s)}\|\Om\|_{1}$ in each $T_j$. There exists a constant $\ka>0$ such that
\Bes
\sum_{v\in\Th_{n+s}}\|(I-G_{n+s,v})T_j^{n,s,v}g_{n-j,s}\|_{L_2(\A)}\lc2^{-(n+s)\ka}\|\Om\|_1\|g_{n-j,s}\|_{L_2(\A)}.
\Ees
\end{lemma}

The proofs of  Lemma \ref{l:19gl2}, Lemma \ref{l:19gl1} will be given in Section \ref{s:196}. Now we use  Lemma \ref{l:19gl2}, Lemma \ref{l:19gl1} to prove \eqref{e:19boffgoal} as follows
\Bes
\begin{split}
&\quad\sum_{s\geq1}\sum_{n\geq100}\Big(\sum_j\|T_jg_{n-j,s}\|^2_{L_2(\A)}\Big)^{1/2}\\
&\leq\sum_{s\geq1}\sum_{n\geq100}\Big(\sum_j\big(\sum_{v\in\Th_{n+s}}\|(I-G_{n+s,v})T_j^{n,s,v}g_{n-j,s}\|_{L_2(\A)}\big)^2\Big)^{1/2}\\
&\quad+\sum_{s\geq1}\sum_{n\geq100}\Big(\sum_j\|\sum_{v\in\Th_{n+s}}G_{n+s,v}T_j^{n,s,v}g_{n-j,s}\|^2_{L_2(\A)}\Big)^{1/2}\lc(\C_\Om\lam\|f\|_{L_1(\A)})^{1/2},
\end{split}
\Ees
where in the second inequality we use Lemma \ref{l:19gl2}, Lemma \ref{l:19gl1} with the fact that for all $s\geq1$,
$\sum_j\|g_{n-j,s}\|^2_{L_2(\A)}\lc\lam\C_\Om^{-1}\|f\|_{L_1(\A)}$ in Lemma \ref{l:19gbasic}. Thus to finish the proof for good functions, it remains to show Lemmas \ref{l:19tgl2}, \ref{l:19gomiota}, \ref{l:19gl2} and \ref{l:19gl1}, which are all given in Section \ref{s:196}.
\begin{remark}
At present, it is easy to see that the proofs for
off-diagonal terms of good functions are similar to that of bad functions. Notice that for the bad functions, we can deal with the diagonal terms (i.e. $s=0$) and the off-diagonals terms (i.e. $s>1$) in a unified way. However this can not be done for the good functions, thus we prove the diagonal terms $g_d$ and the off-diagonal terms $g_\off$ using different method.
The main reason comes from the fact that $g_\off$ has the following property: for all $Q\in\Q_{n-j+s-1}$,
$\int_{Q}g_{n-j,s}(y)dy=0.$
Such kind of cancelation property is crucial in the proof of Lemma \ref{l:19gl1}. But the diagonal terms $g_d$ do not have this cancelation property.
\end{remark}

\vskip0.24cm

\section{Proofs of Lemmas related to the bad functions}\label{s:195}
\vskip0.24cm

In this section, we begin to prove all lemmas of the bad functions in Subsection \ref{s:1942}.
Before that we introduce some lemmas needed in our proof. We first state Schur's Lemma which will be used later.
\begin{lemma}[Schur's lemma]\label{l:19schur}
Suppose that $T$ is an operator with the kernel $K(x,y)$, thus is
\Bes
Tf(x)=\int_{\R^d}K(x,y)f(y)dy.
\Ees
Then $T$ is bounded on $L_2(\A)$ with bound $\sqrt{c_1c_2}$ where
\Bes
c_1=\sup_{x\in\R^d}\int_{\R^d}|K(x,y)|dy,\quad c_2=\sup_{y\in\R^d}\int_{\R^d}|K(x,y)|dx.
\Ees
\end{lemma}
The proof of this lemma could be found in \cite{Par09} or \cite{Gra14C}. We also need the following convexity inequality (or the Cauchy-Schwarz type inequality) for the operator valued function (see \cite[Page 9]{Mei07}). Let $(\mathfrak{m},\mu)$ be a measure space. Suppose that $f:\ \mathfrak{m}\rta\M$ is a weak-* integrable function and $g:\mathfrak{m}\rta\mathbb C$ is an integrable function. Then
\Be\label{e:19conv}
\Big|\int_{\mathfrak{m}}f(x)g(x)d\mu(x)\Big|^2\leq\int_{\mathfrak{m}}|f(x)|^2d\mu(x)\int_{\mathfrak{m}}|g(x)|^2d\mu(x).
\Ee
Below we introduce some basic properties of the bad functions that we will use in our proof.
\begin{lemma}\label{l:19bbasic}
Let $b_{k,s}$ be defined in \eqref{e:19db}. Fix any $s\geq0$. Then we have the following properties for the bad functions $b_{k,s}$:
\begin{enumerate}[\rm (i).]
\item The $L_1$ estimate holds
$
\sum_{k\in\Z}\|b_{k,s}\|_{L_1(\A)}\lc\|f\|_{L_1(\A)};
$
\item For all $k\in\Z$ and $Q\in\Q_{k+s}$, the cancelation property holds
$\int_{Q}b_{k,s}(y)dy=0$.
\end{enumerate}
\end{lemma}

The proof of Lemma \ref{l:19bbasic} could be found in \cite{Cad18} and \cite{Par09}. Now we start to prove Lemmas \ref{l:19cur}, \ref{l:19L^2} and \ref{l:19L^1}.
\vskip0.24cm
\subsection{Proof of Lemma \ref{l:19cur}}\label{s:1951}\quad
\vskip0.24cm
Denote the kernel of the operator $T_{j,\iota}^{n,s}$ by \Be\label{e:19kjiota}
K_{j,\iota}^{n,s}(x-y):=\Om\chi_{D^{\iota}}(\fr{x-y}{|x-y|})K_j(x-y).
\Ee
By the support of $K_j$, it is easy to see that
$$\big\|K_{j,\iota}^{n,s}\big\|_{L_1(\R^d)}\lc\int_{D^{\iota}}\int_{2^{j-1}}^{2^{j+1}}|\Om(\theta)|r^{d-1}2^{-jd}drd\si(\theta)
\lc\int_{D^\iota}|\Om(\theta)|d\si(\theta).$$
Therefore by Chebyshev's inequality, triangle inequality, using H\"older's inequality to remove the projection $\zet$,   we get
\Bes
\tilde{\vphi}\big(\big|\zet\sum_{n\geq100}\sum_{s\geq0}\sum_{j\in
\mathbb{Z}}T_{j,\iota}^{n,s}b_{n-j,s}\eps_j\zet\big|>\lambda\big)\leq\lam^{-1}\sum_{n\geq100}\sum_{s\geq0}\sum_{j\in
\mathbb{Z}}\|T_{j,\iota}^{n,s}b_{n-j,s}\eps_j\|_{L_1(L_\infty(\mathfrak{m})\ot\A)}.
\Ees
Since $\{\eps_j\}_j$ is the Rademacher sequence, thus $\{\eps_j\}_j$ is a bounded sequence. Then we continue to give the above estimate as follows
\begin{equation*}
\begin{split}
\quad\lam^{-1}\sum_{n\geq100}\sum_{s\geq0}\sum_{j\in
\mathbb{Z}}\|T_{j,\iota}^{n,s}b_{n-j,s}\|_{L_1(\A)}
&\lc{\lam}^{-1}\sum_{n\geq100}\sum_{s\geq0}\sum_{j\in
\mathbb{Z}}\|K_{j,\iota}^{n,s}\|_{L_1(\R^d)}\|b_{n-j,s}\|_{L_1(\A)}\\
&\lc{\lam}^{-1}\sum_{n\geq100}\sum_{s\geq0}\int_{D^{\iota}}|\Om(\theta)|d\si(\theta)\sum_{j}\|b_{n-j,s}\|_{L_1(\A)}.\\
\end{split}
\Ees
Now applying the property (i) in Lemma \ref{l:19bbasic}, the above estimate is bounded by
\Bes
\begin{split}
&{\lam}^{-1}\|f\|_{L_1(\A)}\int_{\S^{d-1}}\#\big\{(n,s):\, n\geq 100, s\geq0, 2^{\iota (n+s)}\leq \fr{|\Om(\tet)|}{\|\Om\|_{1}}\big\}|\Om(\tet)|d\si(\tet)\\
&\lc{\lam}^{-1}\|f\|_{L_1(\A)}\int_{\S^{d-1}}|\Om(\tet)|\Big(\big(\log^+(|\Om(\tet)|/\|\Om\|_{1})\big)^2\Big)d\si(\tet)\lc{\lam}^{-1}\C_\Om\|f\|_{L_1(\A)},
\end{split}
\end{equation*}
which completes the proof.
$\hfill{} \Box$
\vskip0.24cm

\subsection {Proof of Lemma \ref{l:19L^2}}\label{s:1952}\quad
\vskip0.24cm
The proof of Lemma \ref{l:19L^2} is based on the following observation of some orthogonality of the support of $\mathcal{F}(G_{k,v})$: For a fixed $k\geq 100$, we have
\begin{equation}\label{e:19obser}
\sup\limits_{\xi\neq0}\sum\limits_{v\in\Theta_{k}}|\Phi^2(2^{k\ga}\inn{e^k_v}{\xi/|\xi|})|\lc2^{k\ga(d-2)}.
\end{equation}
In fact, by homogeneity of $\Phi^2(2^{k\ga}\inn{e^k_v}{\xi/|\xi|})$, it suffices to take the supremum over the surface $\S^{d-1}$. For $|\xi|=1$ and $\xi\in\supp\ \Phi^2(2^{k\ga}\inn{e^k_v}{\xi/|\xi|})$, denote by $\xi^{\bot}$ the hyperplane perpendicular to $\xi$. Then it is easy to see that
\begin{equation}\label{e:19e^n_v}
\text{dist}(e^k_v,\xi^\bot)\lc2^{-k\ga}.
\end{equation}
Since the mutual distance of $e^k_v$'s is bounded by $2^{-k\ga-4}$, there are at most $2^{k\ga(d-2)}$ vectors satisfy
(\ref{e:19e^n_v}). We hence get (\ref{e:19obser}).

Notice that $L_2(\M)$ is a Hilbert space, then the following vector-valued Plancherel's theorem holds
\Bes
\|\mathcal{F}f\|_{L^2(\A)}=(2\pi)^{\fr{d}{2}}\|f\|_{L^2(\A)}=(2\pi)^d\|\mathcal{F}^{-1}f\|_{L^2(\A)}.
\Ees
By applying this Plancherel's theorem, the convex inequality for the operator valued function \eqref{e:19conv}, the fact \eqref{e:19obser} and finally Plancherel's theorem again, we get
\begin{equation}\label{e:19L^2key}
\begin{split}
&\quad\Big\|\sum\limits_{v\in\Theta_{n+s}}G_{n+s,v}T^{n,s,v}_jb_{n-j,s}\Big\|^2_{L_2(\A)}\\
&=(2\pi)^{-\fr{d}{2}}\int_{\R^d}
\tau\Big(\Big|\sum\limits_{v\in\Theta_{n+s}}\Phi(2^{(n+s)\ga}\inn{e^{n+s}_v}{\xi/|\xi|})
\mathcal{F}\big(T^{n,s,v}_jb_{n-j,s}\big)(\xi)\Big|^2\Big)d\xi\\
&\lc\int_{\R^d}\sum\limits_{v\in\Theta_{n+s}}\Phi^2(2^{(n+s)\ga}\inn{e^{n+s}_v}{\xi/|\xi|})
 \tau\Big(\sum\limits_{v\in\Theta_{n+s}}\Big|\mathcal{F}\big(T^{n,s,v}_jb_{n-j,s}\big)(\xi)\Big|^2\Big)d\xi\\
&\lc2^{(n+s)\ga(d-2)} \sum\limits_{v\in\Theta_{n+s}}\big\|T^{n,s,v}_jb_{n-j,s}\big\|^2_{L_2(\A)}.\\
\end{split}
\end{equation}
Once it is showed that for a fixed $e^{n+s}_v$,
\begin{equation}\label{e:19L^2}
\sum_j\Big\| T^{n,s,v}_jb_{n-j,s}\Big\|^2_{L_2(\A)}\lc2^{-2(n+s)\ga(d-1)+2(n+s)\iota}\lam\|\Om\|_{1}\|f\|_{L_1(\A)},
\end{equation}
then by $\card(\Theta_{n+s})\lc2^{(n+s)\ga(d-1)}$, and applying (\ref{e:19L^2key}) and (\ref{e:19L^2}) we get
\begin{equation*}
\begin{split}
\sum\limits_j\Big\|\sum\limits_{v\in\Theta_{n+s}}G_{n+s,v}T^{n,s,v}_jb_{n-j,s}\Big\|^2_{L_2(\A)}\lc2^{-(n+s)\ga+2(n+s)\iota}\lam\C_\Om\|f\|_{L_1(\A)},
\end{split}
\end{equation*}
which is the asserted bound of Lemma \ref{l:19L^2}. Thus, to finish the proof of Lemma \ref{l:19L^2}, it is sufficient to show (\ref{e:19L^2}).

Recall the definition of $b_{n-j,s}$ in \eqref{e:19db}.
By using triangle inequality, to prove \eqref{e:19L^2}, it is enough to prove the following four terms
\Bes
\begin{split}
&\sum_j\Big\| T^{n,s,v}_jp_{n-j}fp_{n-j+s}\Big\|^2_{L_2(\A)}, \quad\sum_j\Big\|T^{n,s,v}_jp_{n-j}f_{n-j+s}p_{n-j+s}\Big\|^2_{L_2(\A)},\\
&\sum_j\Big\|T^{n,s,v}_jp_{n-j+s}fp_{n-j}\Big\|^2_{L_2(\A)}
,\quad\sum_j\Big\|T^{n,s,v}_jp_{n-j+s}f_{n-j+s}p_{n-j}\Big\|^2_{L_2(\A)},
\end{split}
\Ees
satisfy the desired bound in \eqref{e:19L^2}. In the following we will only give the detailed proofs of the first and the second term above, since the proofs of the third and the forth terms are similar.

We first consider the second term which  involves $p_{n-j}f_{n-j+s}p_{n-j+s}$.
Set the kernel of $T_j^{n,s,v}$ as
$
K_j^{n,s,v}(x)=\Ga^{n+s}_v(x)\Om(x)\phi_j(x)|x|^{-d}.
$
By Young's inequality, we get that
\Be\label{e:19bosch}
\|T_j^{n,s,v}p_{n-j}f_{n-j+s}p_{n-j+s}\|^2_{L_2(\A)}\lc\|K_j^{n,s,v}\|_{L_1(\R^d)}^2\|p_{n-j}f_{n-j+s}p_{n-j+s}\|_{L_2(\A)}^2
\Ee

Below we give some estimates for the bound in \eqref{e:19bosch}. Recall that $|\Om(\tet)|\leq 2^{(n+s)\iota}\|\Om\|_{1}$ and the definition of $\Ga_v^{n+s}$ in Subsection \ref{s:1942}. Then by some elementary calculation, we get
\Be\label{e:19Kjnsv}
\|K_j^{n,s,v}\|_{L_1(\R^d)}\lc2^{-(n+s)\ga(d-1)+(n+s)\iota}\|\Om\|_{1}.
\Ee
Notice that $f$ is positive in $\A$. By some basic properties of trace $\vphi$, we write
\Be\label{e:19split}
\begin{split}
&\quad \|p_{n-j}f_{n-j+s}p_{n-j+s}\|_{L_2(\A)}^2=\vphi\big(|p_{n-j}f_{n-j+s}p_{n-j+s}|^2\big)\\
&=\vphi\big(|p_{n-j+s}f_{n-j+s}p_{n-j}|^2\big)=\vphi\big(p_{n-j}f_{n-j+s}p_{n-j+s}f_{n-j+s}p_{n-j}\big)\\
&\leq\vphi\big(p_{n-j}f^{1/2}_{n-j+s}f^{1/2}_{n-j+s}p_{n-j}\big)\cdot\|f^{1/2}_{n-j+s}p_{n-j+s}f^{1/2}_{n-j+s}\|_{\A}.
\end{split}
\Ee
By the trace invariance and modularity of conditional expectations, the first term in the last line above has the following trace preserving property
\Be\label{e:19trace}
\vphi(p_{n-j}f_{n-j+s}p_{n-j})=\vphi(p_{n-j}fp_{n-j})=\vphi(p_{n-j}f).
\Ee
Applying the basic property of $C^*$ algebra, $\|aa^*\|_{\A}=\|a^*a\|_{\A}$, we get that
\Be\label{e:19pnjs}
\begin{split}
\|f^{1/2}_{n-j+s}p_{n-j+s}f^{1/2}_{n-j+s}\|_{\A}&=\|p_{n-j+s}f_{n-j+s}p_{n-j+s}\|_{\A}\\
&=\|p_{n-j+s}q_{n-j+s-1}f_{n-j+s}q_{n-j+s-1}p_{n-j+s}\|_{\A}\\
&\leq2^d\|p_{n-j+s}q_{n-j+s-1}f_{n-j+s-1}q_{n-j+s-1}p_{n-j+s}\|_{\A}\\
&\lc\lam\C_\Om^{-1}
\end{split}
\Ee
where the second equality follows from the identity $p_k=p_kq_{k-1}$ by the definition of $p_k$ and the last inequality follows from $q_kf_kq_k\leq\lam\C_\Om^{-1}q_k$ the property (ii) in Lemma \ref{l:19cucu}.
Now combining \eqref{e:19bosch}, \eqref{e:19Kjnsv}, \eqref{e:19split}, \eqref{e:19trace} and \eqref{e:19pnjs}, we get
\Bes
\begin{split}
\sum_j\|T_j^{n,s,v}p_{n-j}f_{n-j+s}p_{n-j+s}\|^2_{L_2(\A)}
&\lc\C_\Om^{-1}\lam2^{-2(n+s)\ga(d-1)+2(n+s)\iota}\|\Om\|_{1}^2\sum_j\vphi(p_{n-j}f)\\
&\lc\lam2^{-2(n+s)\ga(d-1)+2(n+s)\iota}\|\Om\|_{1}\|f\|_{L_1(\A)}
\end{split}\Ees
which is the required estimate in \eqref{e:19L^2}.

Next we give an estimate of the term corresponding to $p_{n-j}fp_{n-j+s}$.
Notice that there is no average of $f$ in this case and the crucial  property $q_kf_kq_k\leq\lam\C_\Om^{-1}q_k$  can not be applied in the estimate \eqref{e:19pnjs}. Our strategy here is to \emph{add} an average of $f$. In the following we first reduce the proof to the case that the kernel is positive. To do that, we first decompose
\Bes
K_j^{n,s,v}=(K_j^{n,s,v})^+-(K_j^{n,s,v})^-
\Ees
where $(K_j^{n,s,v})^+$ and $(K_j^{n,s,v})^-$ are positive functions. Then by using triangle inequality, we get
\Bes
\begin{split}
\sum_j\Big\| T^{n,s,v}_jp_{n-j}fp_{n-j+s}\Big\|^2_{L_2(\A)}&\lc\sum_j\Big\|\int (K^{n,s,v}_j(\cdot-y))^+p_{n-j}fp_{n-j+s}(y)dy\Big\|^2_{L_2(\A)}\\
&+\sum_j\Big\|\int (K^{n,s,v}_j(\cdot-y))^-p_{n-j}fp_{n-j+s}(y)dy\Big\|^2_{L_2(\A)}.\\
\end{split}
\Ees
Therefore we need to consider the terms related to $(K_j^{n,s,v})^+$ and $(K_j^{n,s,v})^-$, respectively. We only consider the term related to $(K_j^{n,s,v})^+$ since the proof of the other term is similar. For convenience, in the remaining part of this section we still use the abused notation $K_j^{n,s,v}$ to represent $(K_j^{n,s,v})^+$.

Denote the support of $K_j^{n,s,v}$ by $E^{n,s,v}_j$. Then it is not difficult to see that
\Bes
\begin{split}
E_j^{n,s,v}&\subset\{x\in\R^d:|\fr{x}{|x|}-e^{n+s}_v|\leq2^{-(n+s)\ga}, 2^{j-1}\leq|x|\leq2^{j+1}\}\\
&\subset\{x\in \R^d:|\inn{x}{e^{n+s}_v}|\leq 2^{j+1},|x-\inn{x}{e^{n+s}_v}e^{n+s}_v|\leq 2^{j+1-(n+s)\ga}\}.
\end{split}
\Ees

For any $Q\in\Q_{n-j+s}$, let $Q_{n-j}\in\Q_{n-j}$ be the $s$th ancestor of $Q$. By the definition of $p_k$, we may write
\Bes
\begin{split}
T_j^{n,s,v}(p_{n-j}fp_{n-j+s})(x)&=\int_{\R^d}K_j^{n,s,v}(x-y)(p_{n-j}fp_{n-j+s})(y)dy\\
&=\sum_{Q\in\Q_{n-j+s}\atop Q\cap\{x-E_j^{n,s,v}\}\neq\emptyset}\pi_{Q_{n-j}}\Big(\int_{Q}K_j^{n,s,v}(x-y)f(y)dy\Big)\pi_Q\\
&=\sum_{Q\in\Q_{n-j+s}\atop Q\cap\{x-E_j^{n,s,v}\}\neq\emptyset}\int_{Q}\Big[p_{n-j}\big(K_j^{n,s,v}(x-\cdot)f(\cdot)\big)_{n-j+s}p_{n-j+s}\Big](z)dz\\
&=\int_{E_j^{n,s,v}(x)}\big[p_{n-j}\big(K_j^{n,s,v}(x-\cdot)f(\cdot)\big)_{n-j+s}p_{n-j+s}\big](z)dz\\
\end{split}
\Ees
where we use the notation
$$
E_j^{n,s,v}(x)=\bigcup_{Q\in\Q_{n-j+s}\atop Q\cap\{x-E_j^{n,s,v}\}\neq\emptyset}Q.
$$
By the support of $E_j^{n,s,v}$, we see that $E_j^{n,s,v}$ is contained in a rectangle with one sidelength at most $2^{j+1}$ and $d-1$ sidelength at most $2^{j+1-(n+s)\ga}$. Since for any $Q\in \Q_{n-j+s}$, the sidelength $l(Q)=2^{j-(n+s)}\leq 2^{j+1-(n+s)\ga}$. So we get $E_j^{n,s,v}(x)$ is contained in a rectangle with one sidelength at most $2^{j+2}$ and $d-1$ sidelength at most $2^{j+2-(n+s)\ga}$.  Therefore we have the following estimate
\Bes
|E_j^{n,s,v}(x)|\lc2^{jd-{(n+s)}\ga(d-1)}.
\Ees

Next by using the convexity inequality for the operator valued function \eqref{e:19conv} and the preceding inequality, we get
\Bes
\begin{split}
\big|T_j^{n,s,v}(&p_{n-j}fp_{n-j+s})(x)\big|^2\\
&\lc2^{jd-{(n+s)}\ga(d-1)}\int_{E_j^{n,s,v}(x)}\Big|p_{n-j}\Big(K_j^{n,s,v}(x-\cdot)f(\cdot)\Big)_{n-j+s}p_{n-j+s}(z)\Big|^2dz.\\
\end{split}
\Ees

Combining the above estimates, we get that
\Be\label{e:19bl2fn}
\begin{split}
&\|T_j^{n,s,v}p_{n-j}fp_{n-j+s}\|^2_{L_2(\A)}\lc2^{jd-{(n+s)}\ga(d-1)}\\
&\quad\times\int_{\R^d}\int_{E_j^{n,s,v}(x)}\tau\Big(\Big|p_{n-j}\Big(K_j^{n,s,v}(x-\cdot)f(\cdot)\Big)_{n-j+s}p_{n-j+s}(z)\Big|^2\Big)dzdx.
\end{split}
\Ee

Since $K_j^{n,s,v}$ is a positive function and $f$ is a positive operator-valued function in $\A$, we see that $K(x-\cdot)f(\cdot)$ is positive in $\A$. Therefore
\Bes
\begin{split}
\big(K_j^{n,s,v}(x-\cdot)f(\cdot)\big)_{n-j+s}&=\sum_{Q\in\Q_{n-j+s}}\fr{1}{|Q|}\int_{Q}K_j^{n,s,v}(x-y)f(y)dy\chi_{Q}\\
&\lc\sum_{Q\in\Q_{n-j+s}}\fr{1}{|Q|}\int_{Q}f(y)dy\chi_{Q}2^{-jd+(n+s)\iota}\|\Om\|_{1}\\
&=2^{-jd+(n+s)\iota}\|\Om\|_{1}f_{n-j+s}.
\end{split}
\Ees

Now applying the above estimate and  use the same idea in the estimates of \eqref{e:19split} and \eqref{e:19pnjs}, we could get that
\Be\label{e:19tracef}
\begin{split}
&\quad\tau\Big(\Big|p_{n-j}\big(K_j^{n,s,v}(x-\cdot)f(\cdot)\big)_{n-j+s}p_{n-j+s}(z)\Big|^2\Big)\\
&=\tau\Big(p_{n-j}\big( K_j^{n,s,v}(x-\cdot)f(\cdot)\big)_{n-j+s}p_{n-j+s}\big(K_j^{n,s,v}(x-\cdot)f(\cdot)\big)_{n-j+s}p_{n-j}(z)\Big)\\
&\leq\tau\Big(p_{n-j}\big( K_j^{n,s,v}(x-\cdot)f(\cdot)\big)_{n-j+s}p_{n-j}(z)\Big)\\
&\quad\times\|p_{n-j+s}\big(K_j^{n,s,v}(x-\cdot)f(\cdot)\big)_{n-j+s}p_{n-j+s}(z)\|_\M\\
&\lc2^{-2jd+2(n+s)\iota}\|\Om\|^2_{1}\tau\big(p_{n-j}f_{n-j+s}p_{n-j}(z)\big)\|p_{n-j+s}f_{n-j+s}p_{n-j+s}\|_\A\\
&\lc2^{-2jd+2(n+s)\iota}\|\Om\|_{1}\lam\tau\big(p_{n-j}f_{n-j+s}p_{n-j}(z)\big).\\
\end{split}
\Ee

By the definition of $E_j^{n,s,v}(x)$, for any fixed $z\in\R^d$, we have the following estimate
\Be\label{e:19ejnsv}
\Big|\int_{\{x:\,E_j^{n,s,v}(x)\ni z\}}dx\Big|\lc2^{jd-{(n+s)}\ga(d-1)}.
\Ee
Plugging \eqref{e:19tracef} into  \eqref{e:19bl2fn}, then applying Fubini's theorem with \eqref{e:19ejnsv}, and finally using the trace preserving property \eqref{e:19trace}, we  get
\Bes
\begin{split}
\sum_j\|T_j^{n,s,v}p_{n-j}fp_{n-j+s}\|^2_{L_2(\A)}&\lc2^{-2(n+s)\ga (d-1)+2(n+s)\iota}\|\Om\|_{1}\lam\sum_{j\in\Z}\vphi\big(p_{n-j}f_{n-j+s}p_{n-j}\big)\\
&\lc2^{-2(n+s)\ga (d-1)+2(n+s)\iota}\|\Om\|_{1}\lam\|f\|_{L_1(\A)}.
\end{split}
\Ees
Hence, we complete the proof of Lemma \ref{l:19L^2}.
$\hfill{} \Box$
\vskip0.24cm

\subsection{Proof of Lemma \ref{l:19L^1}}\quad
\vskip0.24cm

To prove Lemma \ref{l:19L^1}, we have to face with some oscillatory
integrals which come from $L_j^{n,s,v}$.
Before stating the proof of Lemma \ref{l:19L^1}, let us first give some notation. We introduce the Littlewood-Paley decomposition. Let $\psi$ be a radial
 $C^\infty$ function such that $\psi(\xi)=1$ for $|\xi|\leq 1$, $\psi(\xi)=0$ for $|\xi|\geq 2$
and $0\leq\psi(\xi)\leq1$ for all $\xi\in\R^d$. Define $\beta_k(\xi)=\psi(2^k\xi)-\psi(2^{k+1}\xi)$,
then $\beta_k$ is supported in $\{\xi:2^{-k-1}\leq|\xi|\leq 2^{-k+1}\}$. Choose $\tilde\beta$ be a radial
 $C^\infty$ function such that $\tilde{\beta}(\xi)=1$ for $\fr{1}{2}\leq|\xi|\leq 2$, $\tilde{\beta}$ is supported in $\{\xi:\fr{1}{4}\leq|\xi|\leq4\}$
and $0\leq\tilde{\beta}(\xi)\leq1$ for all $\xi\in\R^d$. Set $\tilde{\beta}_k(\xi)=\tilde{\beta}(2^k\xi)$, then it is easy to see $\beta_k=\tilde{\beta}_k\beta_k$. Define the convolution operators $\Lam_k$ and $\tilde{\Lam}_k$
with the Fourier multipliers $\beta_k$ and $\tilde{\beta}_k$, respectively.
 That is,
$$\widehat{{\Lam}_kf}(\xi)=\beta_k(\xi)\hat{f}(\xi),\quad \ \widehat{\tilde{\Lam}_kf}(\xi)=\tilde{\beta}_k(\xi)\hat{f}(\xi).$$
Then by the construction of $\beta_k$ and $\tilde{\beta}_k$, we have
$\Lam_k=\tilde{\Lam}_k\Lam_k$, $I=\sum\limits_{k\in\mathbb{Z}}\Lam_k.$
Write
$L_j^{n,s,v}=\sum\limits_{k}(I-G_{n+s,v})\Lam_{k}T_{j}^{n,s,v}.$
Then triangle inequality gives us
\Bes
\|L_j^{n,s,v}b_{n-j,s}\|_{L_1(\A)}\leq\sum_{k\in\Z}\|(I-G_{n+s,v})\Lam_{k}T_{j}^{n,s,v}b_{n-j,s}\|_{L_1(\A)}.
\Ees
In the remaining part of this subsection,
we show that two different estimates could be established for $\|(I-G_{n+s,v})\Lam_{k}T_{j}^{n,s,v}b_{n-j,s}\|_{L_1(\A)}$, which will deduce Lemma \ref{l:19L^1} by taking a sum over $k\in\Z$ with these two different estimates.
\begin{lemma}\label{l:19l^1_1}
With all notation above. Then there exists $N>0$ such that the following estimate holds
\Be\label{e:19main}
\begin{split}
\|(I-G_{n+s,v})&\Lam_kT_j^{n,s,v}b_{n-j,s}\|_{L_1(\A)}\\
&\lc2^{-(n+s)\ga(d-1)+(n+s)\iota+(k-j)+(n+s)\ga(1+2N)}\|\Om\|_1\|b_{n-j,s}\|_{L_1(\A)}.
\end{split}
\Ee
\end{lemma}

\begin{proof}
Applying Fubini's theorem, we may write
\Be\label{e:19k11}
(I-G_{n+s,v})\Lam_k T_{j}^{n,s,v}b_{n-j,s}(x)=:\int_{\R^d} D_{k}^{n,s,v}(x-y)b_{n-j,s}(y)dy
\Ee
where $D_k^{n,s,v}(x)$ is defined as the kernel of the operator $(I-G_{n+s,v})\Lam_{k}T^{n,s,v}_{j}$. More precisely, $D_k^{n,s,v}$ can  be written as
\begin{equation}\label{e:19dkde}
D_k^{n,s,v}(x)=\fr{1}{(2\pi)^{d}}
\int_{\R^d} e^{ix\cdot\xi}h_{k,n,s,v}(\xi)\int_{\R^d} e^{-i\xi\cdot\om}\Om(\om)\Ga_v^{n+s}(\om)K_j(\om) d\om d\xi,
\end{equation}
where $h_{k,n,s,v}(\xi)=(1-\Phi(2^{(n+s)\ga}\inn{e^{n+s}_v}{\xi/|\xi|}))\beta_{k}(\xi).$
Using Young's inequality, we get
\Bes
\|(I-G_{n+s,v})\Lam_kT_j^{n,s,v}b_{n-j,s}\|_{L_1(\A)}
\leq\|D_{k}^{n,s,v}\|_{L_1(\R^d)}\|b_{n-j,s}\|_{L_1(\A)}.
\Ees

Hence in the following we only need to give a $L_1$ estimate of $D_k^{n,s,v}$. In order to separate the rough kernel, we make a change of variable $\om=r\theta$. By Fubini's theorem, $D_k^{n,s,v}(x)$ can be written as
\begin{equation}\label{e:19mainintegral}
\fr{1}{(2\pi)^{d}}\int_{\S^{d-1}}\Om(\theta)\Ga_v^{n+s}(\theta)
\bigg\{\int_{\R^d}\int_0^\infty
e^{i\inn{x-r\theta}{\xi}}h_{k,n,s,v}(\xi)K_j(r)r^{d-1} drd\xi\bigg\}d\si(\tet).
\end{equation}
Concerning the support of $K_j$, we have $2^{j-1}\leq r\leq2^{j+1}$. Integrating by parts with $r$, then the integral involving $r$ can
be rewritten as
$$\int_0^\infty e^{-i\inn{r\theta}{\xi}}(i\inn{\theta}{\xi})^{-1}
\pari_r[K_j(r)r^{d-1}]dr.$$

Since $\theta\in\supp\Ga^{n+s}_v$, then $|\theta-e^{n+s}_v|\leq 2^{-(n+s)\ga}$. By the support of $\Phi$,
we see $|\inn{e^{n+s}_v}{\xi/|\xi|}|\geq 2^{1-(n+s)\ga}$. Thus,
\begin{equation}\label{e:192ang}
|\inn{\theta}{\xi/|\xi|}|\geq|\inn{e^{n+s}_v}{\xi/|\xi|}|-|\inn{e^{n+s}_v-\theta}{\xi/|\xi|}|\geq2^{-(n+s)\ga}.
\end{equation}
After integrating by parts with $r$, integrating by parts $N$ times with $\xi$, then the integral in \eqref{e:19mainintegral}
can be rewritten as
\begin{equation}\label{e:19minte}
\begin{split}
\fr{1}{(2\pi)^d}\int_{\S^{d-1}}&\Om(\tet)\Ga_v^{n+s}(\tet)\int_{\R^d}\int_0^\infty e^{i\inn{x-r\tet}{\xi}}
\pari_r\big[K_j(r)r^{d-1}\big]\\
&\times\fr{(I-2^{-2k}\Delta_\xi)^N}{(1+2^{-2k}|x-r\tet|^2)^N}\Big(h_{k,n,s,v}(\xi)(i\inn{\tet}{\xi})^{-1}\Big)drd\xi d\si(\tet).
\end{split}
\end{equation}

In the following, we give explicit estimates of all terms in (\ref{e:19minte}).
We show that the following estimate holds
\begin{equation}\label{e:192D2}
\big|(I-2^{-2k}\Delta_\xi)^{N}[\inn{\theta}{\xi}^{-1}h_{k,n,s,v}(\xi)]\big|\lc2^{(n+s)\ga+k+2(n+s)\ga N}.
\end{equation}
Firstly we prove \eqref{e:192D2} when $N=0$. By \eqref{e:192ang}, we have
$$|(-i\inn{\theta}{\xi})^{-1}\cdot h_{k,n,s,v}(\xi)|\lc|\inn{\theta}{\xi}|^{-1}\lc2^{(n+s)\ga+k}.$$
Next we consider $N=1$ in \eqref{e:192D2}. By using product rule and some elementary calculation, we get that
\begin{equation*}
\begin{split}
|\pari_{\xi_i}h_{k,n,s,v}(\xi)|&\leq\big|-\pari_{\xi_i}[\Phi(2^{(n+s)\ga}\inn{e^{n+s}_v}{\xi/|\xi|})]
\cdot\beta_{k}(\xi)\big|\\
&\quad+\big|\pari_{\xi_i}\beta_{k}(\xi)\cdot
(1-\Phi(2^{(n+s)\ga}\inn{e^{n+s}_v}{\xi/|\xi|}))\big|\lc2^{(n+s)\ga+k}.
\end{split}
\end{equation*}
Therefore by induction, we have
$|\pari_{\xi}^{\alpha}h_{k,n,s,v}(\xi)|\lc2^{((n+s)\ga+k)|\alpha|}$
for any  multi-indices $\alpha\in\Z^n_+$.
By using product rule again and (\ref{e:192ang}), we have
\begin{equation*}
\begin{split}
\big|\pari^2_{\xi_i}(\inn{\theta}{\xi})^{-1}h_{k,n,s,v}(\xi))\big|&\leq\big|{2\inn{\theta}{\xi}}^{-3}\cdot \theta_i^2\cdot h_{k,n,s,v}(\xi)\big|+\big|2{\inn{\theta}{\xi}}^{-2}\cdot\theta_i\pari_{\xi_i}h_{k,n,s,v}(\xi)\big|\\
&\quad+\big|\inn{\theta}{\xi}^{-1}\pari_{\xi_i}^2h_{k,n,s,v}(\xi)\big|\lc2^{3((n+s)\ga+k)}.\\
\end{split}
\end{equation*}
Hence we conclude that
$
2^{-2k}\big|\Delta_\xi[(\inn{\theta}{\xi})^{-1}h_{k,n,s,v}(\xi)]\big|\lc2^{(n+s)\ga+k+2(n+s)\ga}.
$
Proceeding by induction, we get \eqref{e:192D2}.

By the definition of $K_j$ and using  product rule, it is not difficult to get
\begin{equation}\label{e:19parir}
\Big|\pari_r\big(K_j(r)r^{d-1}\big)\Big|\lc2^{-2j}.
\end{equation}

Now we choose $N=[d/2]+1$.
Since we need to get the $L^1$ estimate of (\ref{e:19minte}), by the support of $h_{k,n,s,v}$, $|\xi|\approx 2^{-k}$, then
$$\int_{|\xi|\approx 2^{-k}}\int_{\R^d}\Big(1+2^{-2k}|x-r\theta|^2\Big)^{-N}dxd\xi\leq C.$$
Now combine (\ref{e:19parir}), (\ref{e:192D2}) and above estimates. Next integrating with $r$, we get a bound $2^j$. Note that we suppose that $\|\Om\|_\infty\leq 2^{(n+s)\iota}\|\Om\|_1$. Then integrating with $\theta$, we get a bound $2^{-(n+s)\ga(d-1)+(n+s)\iota}\|\Om\|_1$.
So we finally get that
\Be\label{e:19dknsv}
\begin{split}
\|D_{k}^{n,s,v}\|_{L_1(\R^d)}&\lc 2^{-2j+(n+s)\ga+k+2(n+s)\ga N+j-(n+s)\ga(d-1)+(n+s)\iota}\|\Om\|_1\\
&=2^{-(n+s)\ga(d-1)+(n+s)\iota-j+k+(n+s)\ga(1+2N)}\|\Om\|_1.\\
\end{split}
\Ee
Hence we complete the proof of Lemma \ref{l:19l^1_1} with $N=[\fr{d}{2}]+1$.
\end{proof}

\begin{lemma}\label{l:19bl1can}
With all notation above, the following estimate holds
\Bes
\begin{split}
\|(I-G_{n+s,v})\Lam_k&T_j^{n,s,v}b_{n-j,s}\|_{L_1(\A)}\\
&\lc2^{-(n+s)\ga (d-1)-(n+s)+j-k+(n+s)\iota}\|\Om\|_1\|b_{n-j,s}\|_{L_1(\A)}.
\end{split}
\Ees
\end{lemma}
\begin{proof}
Using $\Lam_k=\Lam_k\tilde{\Lam}_k$, we write
\Bes
\begin{split}
\|(I-G_{n+s,v})&\Lam_kT_j^{n,s,v}b_{n-j,s}\|_{L_1(\A)}=\|(I-G_{n+s,v})\tilde{\Lam}_k\Lam_kT_j^{n,s,v}b_{n-j,s}\|_{L_1(\A)}\\
&\lc\|(I-G_{n+s,v})\tilde{\Lam}_k\|_{L_1(\A)\rta L_1(\A)}\|\Lam_kT_j^{n,s,v}b_{n-j,s}\|_{L_1(\A)}.
\end{split}
\Ees
Then it is easy to see that the proof of this lemma follows from the following two estimates:
\Be\label{e:19l1l1}
\|(I-G_{n+s,v})\tilde{\Lam}_k\|_{L_1(\A)\rta L_1(\A)}\lc1
\Ee
and
\Be\label{e:19bjcan}
\|\Lam_kT_j^{n,s,v}b_{n-j,s}\|_{L_1(\A)}\lc2^{-(n+s)\ga (d-1)-(n+s)+j-k+(n+s)\iota}\|\Om\|_1\|b_{n-j,s}\|_{L_1(\A)}.
\Ee

We first consider the estimate \eqref{e:19l1l1}. The kernel of $(I-G_{n+s,v})\tilde{\Lam}_k$ is the inverse Fourier transform of  $\tilde{h}_{k,n,s,v}(\xi)=[1-\Phi(2^{(n+s)\ga}\inn{e^{n+s}_v}{\xi/|\xi|})]\tilde{\beta}_k(\xi)$. So
$$\|(I-G_{n+s,v})\tilde{\Lam}_k\|_{L_1(\A)\rta L_1(\A)}\lc\|\mathcal{F}(\tilde{h}_{k,n,s,v})\|_{L^1(\R^d)}=\big\|\mathcal{F}\big[\tilde{h}_{k,n,s,v}(A_k^{n,s,v}\cdot)\big]\big\|_{L_1(\R^d)}$$
where $A_k^{n,s,v}$ is an invertible linear transform such that $A_k^{n,s,v}e^{n+s}_v=2^{-(n+s)\ga-k}e^{n+s}_v$ and $A_k^{n,s,v}y=2^{-k}y$ if $\inn{y}{e^{n+s}_v}=0$. For all $\alp\in\Z^d_+$, it is straightforward to check that
$$\big\|\pari^{\alp}\big[\tilde{h}_{k,n,s,v}(A_k^{n,s,v}\cdot)\big]\big\|_{L_2(\R^d)}\lc C_\alp$$
uniformly with $k,n,s,v$ (see \cite[Page 100]{See96}). Therefore splitting the following integral into two parts and using Plancherel's theorem, we get
\Bes
\begin{split}
&\big\|\mathcal{F}\big[\tilde{h}_{k,n,s,v}(A_k^{n,s,v}\cdot)\big]\big\|_{L^1(\R^d)}=\Big(\int_{|\xi|\geq1}+\int_{|\xi|<1}\Big)\big|\mathcal{F}\big[\tilde{h}_{k,n,s,v}(A_k^{n,s,v}\cdot)\big](\xi)\big|d\xi\\
&\lc\Big(\int_{|\xi|\geq1}\fr{d\xi}{|\xi|^{2([\fr{d}{2}]+1)}}\Big)^{\fr{1}{2}}\sum_{|\alp|=[\fr{d}{2}]+1}\Big(\int_{\R^d}|\xi^{\alp}\mathcal{F}\big[\tilde{h}_{k,n,s,v}(A_k^{n,s,v}\cdot)\big](\xi)|^2d\xi\Big)^{1/2}\\
&\quad+\big\|\mathcal{F}\big[\tilde{h}_{k,n,s,v}(A_k^{n,s,v})\big]\big\|_{L_2(\R^d)}\\
&\lc\sum_{|\alp|=[\fr{d}{2}]+1}\big\|\pari^{\alp}\big[\tilde{h}_{k,n,s,v}\big(A_k^{n,s,v}\cdot\big)\big]\big\|_{L_2(\R^d)}+\|\tilde{h}_{k,n,s,v}(A_k^{n,s,v}\cdot)\|_{L_2(\R^d)}\lc1,
\end{split}
\Ees
which completes the proof of \eqref{e:19l1l1}.

Now we turn to another estimate \eqref{e:19bjcan}. Write
\Bes
\Lam_kT_j^{n,s,v}b_{n-j,s}=\check{\beta}_k*K_j^{n,s,v}*b_{n-j,s}=K_j^{n,s,v}*\check{\beta}_k*b_{n-j,s}.
\Ees
Then by the estimate \eqref{e:19Kjnsv} of $K_j^{n,s,v}$, we get that
\Be\label{e:19btjcan}
\begin{split}
\|\Lam_kT_j^{n,s,v}b_{n-j,s}\|_{L_1(\A)}&\leq\|K_j^{n,s,v}\|_{L_1(\R^d)}\|\check{\beta}_k*b_{n-j,s}\|_{L_1(\A)}\\
&\lc2^{-(n+s)(\ga(d-1)-\iota)}\|\Om\|_{1}\|\check{\beta}_k*b_{n-j,s}\|_{L_1(\A)}.
\end{split}
\Ee
Note that $\beta_k(\xi)=\beta(2^k\xi)$, we get $\check{\beta}_k(x)=2^{-kd}\check{\beta}(2^{-k}x)$. Therefore we see
\Be\label{e:19betak}
\int_{\R^d}|\nabla[\check{\beta}_k](x)|dx=2^{-k(d+1)}\int_{\R^d}|\nabla(\check{\beta})(2^{-k}x)|dx=2^{-k}\int_{\R^d}|\nabla(\check{\beta})(x)|dx.
\Ee

Using the cancelation property (ii) in Lemma \ref{l:19bbasic}, we see that for all $Q\in\Q_{n-j+s}$,
$
\int_{Q}b_{n-j,s}(y)dy=0.
$
Let $y_Q$ be the center of $Q$. Notice that for all $y\in Q$, $|y-y_Q|\lc2^{j-n-s}$. Using this cancelation property, we then get
\Bes
\begin{split}
&\quad\|\check{\beta}_k*b_{n-j,s}\|_{L_1(\A)}\\
&=\int_{\R^d}\tau\Big(\Big|\sum_{Q\in\Q_{n-j+s}}\int_Q[\check{\beta}_k(x-y)-\check{\beta}_k(x-y_Q)]b_{n-j,s}(y)dy\Big|\Big)dx\\
&\leq\int_{\R^d}\sum_{Q\in\Q_{n-j+s}}\int_Q\Big|\int_0^1\inn{y-y_Q}{\nabla {[\check{\beta}_k]}(x-\rho y-(1-\rho)y_Q)}d\rho\Big|\tau(|b_{n-j,s}(y)|)dydx\\
&\lc2^{j-n-s-k}\|b_{n-j,s}\|_{L_1(\A)},
\end{split}
\Ees
where in the second inequality we just use the mean value formula.
Combining this inequality with \eqref{e:19btjcan} yields the estimate \eqref{e:19bjcan}. Hence we finish the proof of this lemma.
\end{proof}

Now we conclude the proof of Lemma \ref{l:19L^1} as follows. Let $\eps_0\in(0,1)$ be a constant which will be chosen later. Notice that $\card(\Theta_{n+s})\lc2^{(n+s)\ga(d-1)}$.
Then by Lemma \ref{l:19l^1_1} with $N={[d/2]+1}$, Lemma \ref{l:19bl1can} and the property $\sum_j\|b_{n-j,s}\|_{L_1(\A)}\lc\|f\|_{L_1(\A)}$ in Lemma \ref{l:19bbasic}, we get that
\Bes
\begin{split}
&\quad\sum_j\sum_{v\in\Theta_{n+s}}\|L_j^{n,s,v}b_{n-j,s}\|_{L_1(\A)}\\
&\leq\sum_j\sum_{v\in\Theta_{n+s}}\Big(\sum_{k\leq j-[(n+s)\eps_0]}+\sum_{k\geq j-[(n+s)\eps_0]}\Big)\|(I-G_{n+s,v})\Lam_kT_j^{n,s,v}b_{n-j,s}\|_{L_1(\A)}\\
&\lc\sum_j\Big(2^{-(n+s)(\eps_0-\ga(3+2[\fr{d}{2}])-\iota)}+
2^{-{(n+s)(1-\eps_0-\iota)}}\Big)\|b_{n-j,s}\|_{L_1(\A)}\lc 2^{-(n+s)\alp}\|f\|_{L_1(\A)}
\end{split}
\Ees
where we choose the constants $0<\iota\ll\ga\ll\eps_0\ll1$ such that the constant $\alp$ is defined by
$$\alp=\min\{\eps_0-\ga(3+2[\fr{d}{2}])-\iota,1-\eps_0-\iota\}>0.$$
Hence we complete the proof.
$\hfill{} \Box$
\vskip0.24cm

\section{Proofs of Lemmas related to the good functions}\label{s:196}
\vskip0.24cm
In this section, we begin to prove all lemmas of the good functions in Subsection \ref{s:1943}.
The proofs for off-diagonal terms are similar to those for bad functions in Section \ref{s:195}, so we shall be brief and only indicate necessary changes in the proofs of off-diagonal terms. We first consider the proofs of diagonal terms.

\subsection{Proof of Lemma \ref{l:19tgl2}}\label{s:1961}\quad
\vskip0.24cm

Recall the definition of $T$. Let $\mathbf{K}_j$ be the kernel of the operator $T_j$, i.e. $\mathbf{K}_j(x)=\Om(x)\phi_j(x)|x|^{-d}$.  Notice that $\{\eps_j\}_j$ is a Rademacher sequence on a probability space $(\mathfrak{m},P)$,  then applying the equality \eqref{e:19orth}, we write
\Bes
\|Tf\|_{L_2(L_\infty(\mathfrak{m}\ot\A))}^2=\sum_{j\in\Z}\|T_jf\|_{L_2(\A)}^2=(2\pi)^{-d/2}\int_{\R^d}\sum_{j\in\Z}|\widehat{\mathbf K_j}(\xi)|^2\tau(|\hat{f}(\xi)|^2)d\xi
\Ees
where the second equality follows from Plancherel's theorem since $L_2(\M)$ is a Hilbert space. In the following we show that
\Be\label{e:19plan}
\sum_{j\in\Z}|\widehat{\mathbf K_j}(\xi)|^2<\infty
\Ee
holds for almost every $\xi\in\R^d$. Once we prove the inequality \eqref{e:19plan}, Lemma \ref{l:19tgl2} follows from Plancherel's theorem. Now we fix $\xi\neq0$. By the cancelation property of $\Om$, $\int_{\S^{d-1}}\Om(\tet)d\si(\tet)=0$, we get that
\Bes
|\widehat{\mathbf K_j}(\xi)|=\Big|\int_{\R^d}K_j(x)(e^{-i\xi x}-1)dx\Big|\lc2^j|\xi|\|\Om\|_1.
\Ees
Therefore the sum over all $j$s satisfying $2^j|\xi|\leq1$ is convergent.

Now we turn to the case $2^j|\xi|>1$. We split the kernel $\Om(\tet)$ into two parts: $\Om_1(\tet)=\Om(\tet)\chi_{\{\tet\in\S^{d-1}:|\Om(\tet)|\leq2^{j\nu}|\xi|^\nu\|\Om\|_1\}}$
and $1-\Om_1(\tet)$ for some constant $\nu\in(0,1/2)$. We first consider $\Om_1$. By making a change of variable $x=r\tet$, we get
\Be\label{e:19kjpl}
|\widehat{\mathbf K_j}(\xi)|\leq\int_{\S^{d-1}}|\Om_1(\tet)|\Big|\int_{\R}e^{-ir\inn{\tet}{\xi}}\phi_j(r)r^{-1}dr\Big|d\si(\tet).
\Ee

It is easy to see that $|\int_{\R}e^{-ir\inn{\tet}{\xi}}\phi_j(r)r^{-1}dr|$ is finite. By integrating by parts with the variable $r$, we get that
\Bes
\Big|\int_{\R}e^{-ir\inn{\tet}{\xi}}\phi_j(r)r^{-1}dr\Big|
=\Big|\int_{\R}e^{-ir\inn{\tet}{\xi}}{\inn{\tet}{\xi}}^{-1}\pari_r[\phi_j(r)r^{-1}]dr\Big|\lc(2^j|\xi|)^{-1}|\inn{\tet}{\xi'}|^{-1}
\Ees
where $\xi'=\xi/|\xi|$. Interpolating these two estimates we get that for any $\del\in(1/2,1)$,
\Bes
\Big|\int_{\R}e^{-ir\inn{\tet}{\xi}}\phi_j(r)r^{-1}dr\Big|\lc(2^j|\xi|)^{-\del}|\inn{\tet}{\xi'}|^{-\del}.
\Ees
Plugging the above estimate into \eqref{e:19kjpl} with the fact $\int_{\S^{d-1}}|\inn{\tet}{\xi'}|^{-\del}d\si(\tet)<\infty$, we hence get that
$
|\widehat{\mathbf K_j}(\xi)|\lc(2^j|\xi|)^{-\del+\nu}\|\Om\|_1,
$
which is sufficient for us taking a sum over all $j$s satisfying $2^j|\xi|>1$. Consider the other term $1-\Om_1$. Then we get
\Bes
\begin{split}
&\sum_{j:\ 2^j|\xi|>1}|\widehat{\mathbf K_j}(\xi)|^2
\lc\sum_{j:\ 2^j|\xi|>1}\Big(\int_{\{\tet\in\S^{d-1}:|\Om(\tet)|\geq(2^j|\xi|)^{\nu}\|\Om\|_1\}}|\Om(\tet)|d\si(\tet)\Big)^2\\
&=\int_{\S^{d-1}\times\S^{d-1}}\#\big\{j:1<2^j|\xi|\leq\min\{(|\Om(\tet)|/\|\Om\|_1)^{\fr{1}{\nu}},(|\Om(\alp)|/\|\Om\|_1)^{\fr{1}{\nu}}\}\big\}\\
&\quad\quad\times|\Om(\tet)|\cdot|\Om(\alp)|d\si(\tet)d\si(\alp)\\
&\lc\Big(\int_{\S^{d-1}}|\Om(\tet)|(1+[\log^+(|\Om(\tet)/\|\Om\|_1)]^{1/2})d\si(\tet)\Big)^2<\infty
\end{split}
\Ees
where the last inequality just follows from $\Om\in L(\log^+ L)^{1/2}(\S^{d-1})$. Hence we complete the proof.
$\hfill{} \Box$
\vskip0.24cm

\subsection{Proof of Lemma \ref{l:19gomiota}}\quad
\vskip0.24cm

Recall the definition of the kernel $K_{j,\iota}^{n,s}$ in \eqref{e:19kjiota}. By Young's inequality, it is easy to see that
\Bes
\|T_{j,\iota}^{n,s}g_{n-j,s}\|_{L_2(\A)}\leq\|K_{j,\iota}^{n,s}\|_{L_1(\R^d)}\|g_{n-j,s}\|_{L_2(\A)}\lc\int_{D^{\iota}}|\Om(\tet)|d\si(\tet)\|g_{n-j,s}\|_{L_2(\A)}.
\Ees
Now applying $\sum_j\|g_{n-j,s}\|_{L_2(\A)}^2\lc\lam\C_\Om^{-1}\|f\|_{L_1(\A)}$ in Lemma \ref{l:19gbasic} and the above estimate, we get that
\Bes
\begin{split}
&\quad\sum_{s\geq1}\sum_{n\geq100}\Big(\sum_j\|T_{j,\iota}^{n,s}g_{n-j,s}\|^2_{L_2(\A)}\Big)^{\fr{1}{2}}
\lc\sum_{s\geq1}\sum_{n\geq100}\int_{D^{\iota}}|\Om(\tet)|d\si(\tet)\Big(\lam\C_\Om^{-1}\|f\|_{L_1(\A)}\Big)^{1/2}\\
&\lc\int_{\S^{d-1}}\#\{(s,n):s\geq1,n\geq100,|\Om(\tet)|\geq2^{(n+s)\iota}\|\Om\|_1\}|\Om(\tet)|d\si(\tet)\Big(\lam\C_\Om^{-1}\|f\|_{L_1(\A)}\Big)^{1/2}\\
&\lc\Big(\C_\Om\lam\|f\|_{L_1(\A)}\Big)^{1/2},
\end{split}
\Ees
which is our desired estimate. Hence we complete  the proof.
$\hfill{} \Box$
\vskip0.24cm

\subsection{Proof of Lemma \ref{l:19gl2}}\quad
\vskip0.24cm

By applying Plancherel's theorem, the convex inequality for the operator valued function \eqref{e:19conv},
the fact \eqref{e:19obser} and finally Plancherel's theorem again, we get
\begin{equation}\label{e:19gL^2key}
\begin{split}
&\quad\Big\|\sum\limits_{v\in\Theta_{n+s}}G_{n+s,v}T^{n,s,v}_jg_{n-j,s}\Big\|^2_{L_2(\A)}\\
&=(2\pi)^{-\fr{d}{2}}\int_{\R^d}
\tau\Big(\Big|\sum\limits_{v\in\Theta_{n+s}}\Phi(2^{(n+s)\ga}\inn{e^{n+s}_v}{\xi/|\xi|})
\mathcal{F}\big(T^{n,s,v}_jg_{n-j,s}\big)(\xi)\Big|^2\Big)d\xi\\
&\lc\int_{\R^d}\sum\limits_{v\in\Theta_{n+s}}\Phi^2(2^{(n+s)\ga}\inn{e^{n+s}_v}{\xi/|\xi|})
 \sum\limits_{v\in\Theta_{n+s}}\tau\Big(\Big|\mathcal{F}\big(T^{n,s,v}_jg_{n-j,s}\big)(\xi)\Big|^2\Big)d\xi\\
&\lc2^{(n+s)\ga(d-2)} \sum\limits_{v\in\Theta_{n+s}}\big\|T^{n,s,v}_jg_{n-j,s}\big\|^2_{L_2(\A)}.\\
\end{split}
\end{equation}

Using Young's inequality and \eqref{e:19Kjnsv}, we get that $
\big\|T^{n,s,v}_jg_{n-j,s}\big\|^2_{L_2(\A)}$ is bounded by
\Be\label{e:19gl2kj}
\|K_j^{n,s,v}\|_{L_1(\R^d)}^2\|g_{n-j,s}\|^2_{L_2(\A)}\lc2^{2(n+s)(-\ga(d-1)+\iota)}\|\Om\|^2_1\|g_{n-j,s}\|^2_{L_2(\A)}.
\Ee

Now plugging \eqref{e:19gl2kj} into \eqref{e:19gL^2key} and using the fact $\card(\Theta_{n+s})\lc2^{(n+s)\ga(d-1)}$, we get that
\Bes
\Big\|\sum_{v\in\Theta_{n+s}}G_{n+s,v}T^{n,s,v}_jg_{n-j,s}\Big\|^2_{L_2(\A)}\lc2^{(n+s)(-\ga+2\iota)}\|\Om\|^2_1\|g_{n-j,s}\|^2_{L_2(\A)}
\Ees
which is just our desired estimate.
$\hfill{} \Box$
\vskip0.24cm

\subsection{Proof of Lemma \ref{l:19gl1}}\quad
\vskip0.24cm
Using $I=\sum_k\Lam_k$ and triangle inequality, we get that
\Bes
\|(I-G_{n+s,v})T_j^{n,s,v}g_{n-j,s}\|_{L_2(\A)}\leq\sum_k
\|(I-G_{n+s,v})\Lam_kT_j^{n,s,v}g_{n-j,s}\|_{L_2(\A)}
\Ees
Let $\eps_0\in(0,1)$ be a constant which will be chosen later. Separating the above sum into two parts, we will prove that
\Be\label{e:19b1kgeq}
\begin{split}
\sum_{k\leq j-[(n+s)\eps_0]}\|(I-G_{n+s,v})&\Lam_kT_j^{n,s,v}g_{n-j,s}\|_{L_2(\A)}\\
&\lc2^{-(n+s)(\ga(d-1)+\eps_0-\ga(3+2[\fr{d}{2}])-\iota)}\|\Om\|_1\|g_{n-j,s}\|_{L_2(\A)},
\end{split}
\Ee
and
\Be\label{e:19b1kle}
\begin{split}
\sum_{k> j-[(n+s)\eps_0]}\|(I-G_{n+s,v})&\Lam_kT_j^{n,s,v}g_{n-j,s}\|_{L_2(\A)}\\
&\lc2^{-(n+s)(\ga(d-1)+1-\eps_0-\iota)}\|\Om\|_1\|g_{n-j,s}\|_{L_2(\A)}.
\end{split}
\Ee

Based on \eqref{e:19b1kgeq}, \eqref{e:19b1kle} and the fact $\card(\Theta_{n+s})\lc2^{(n+s)\ga(d-1)}$, we finish the proof of this lemma by choosing the constants $0<\iota\ll\ga\ll\eps_0\ll1$ such that the constant $\ka$ defined by
$$\ka=\min\{\eps_0-\ga(3+2[\fr{d}{2}])-\iota,1-\eps_0-\iota\}>0.$$

Now we give the proof of \eqref{e:19b1kgeq} and \eqref{e:19b1kle}. Consider \eqref{e:19b1kgeq} first. Recall that $D_k^{n,s,v}(x)$ is defined as the kernel of the operator $(I-G_{n+s,v})\Lam_{k}T^{n,s,v}_{j}$ in \eqref{e:19dkde}. Applying Young's inequality and the estimate of $D_k^{n,s,v}$ in \eqref{e:19dknsv}, we get
\Bes
\begin{split}
\|(I-G_{n+s,v})\Lam_kT_j^{n,s,v}g_{n-j,s}\|_{L_2(\A)}&\leq\|D_k^{n,s,v}\|_{L_1(\R^d)}\|g_{n-j,s}\|_{L_2(\A)}\\
&\lc2^{-(n+s)(\ga(d-1)-\iota-(3+2[\fr{d}{2}])\ga)-j+k}\|\Om\|_1\|g_{n-j,s}\|_{L_2(\A)}.
\end{split}
\Ees
Taking a sum over $k\leq j-[(n+s)\eps_0]$ yields \eqref{e:19b1kgeq}.

Next we turn to the proof of \eqref{e:19b1kle}. By Plancherel's theorem, we see that
\Be\label{e:19ignv}
\begin{split}
\|(I-G_{n+s,v})\Lam_kT_j^{n,s,v}g_{n-j,s}\|_{L_2(\A)}\lc\|\Lam_kT_j^{n,s,v}g_{n-j,s}\|_{L_2(\A)}.
\end{split}
\Ee
Write
\Bes
\Lam_kT_j^{n,s,v}g_{n-j,s}=\check{\beta}_k*K_j^{n,s,v}*g_{n-j,s}=K_j^{n,s,v}*\check{\beta}_k*g_{n-j,s}.
\Ees
Then by Young's inequality and \eqref{e:19Kjnsv}, we get that
\Be\label{e:19gtjcan}
\begin{split}
\|\Lam_kT_j^{n,s,v}g_{n-j,s}\|_{L_2(\A)}&\leq\|K_j^{n,s,v}\|_{L_1(\R^d)}\|\check{\beta}_k*g_{n-j,s}\|_{L_2(\A)}\\
&\lc2^{-(n+s)(\ga(d-1)-\iota)}\|\Om\|_{1}\|\check{\beta}_k*g_{n-j,s}\|_{L_2(\A)}.
\end{split}
\Ee

Recall the definition of $g_{n-j,s}$, we have the following cancelation property:  for all $s\geq1$ and $Q\in\Q_{n-j+s-1}$,
$
\int_{Q}g_{n-j,s}(y)dy=0.
$
Let $y_Q$ be the center of $Q$. Using this cancelation property, we get
\Bes
\begin{split}
\check{\beta}_k*g_{n-j,s}(x)&=\int_{\R^d}\sum_{Q\in\Q_{n-j+s-1}}[\check{\beta}_k(x-y)-\check{\beta}_k(x-y_Q)]\chi_{Q}(y)g_{n-j,s}(y)dy\\
&=:\int_{\R^d}K_k(x,y)g_{n-j,s}(y)dy,
\end{split}
\Ees
with $K_k(x,y)=\sum_{Q\in\Q_{n-j+s-1}}[\check{\beta}_k(x-y)-\check{\beta}_k(x-y_Q)]\chi_{Q}(y)$. Below we will apply Schur's lemma to give an estimate of $\|\check{\beta}_k*g_{n-j,s}\|_{L_2(\A)}$.
We first consider $K_k(x,y)$ as follows: For any $y$, there exists a unique cube $Q\in\Q_{n-j+s-1}$ such that $y\in Q$.
Then by \eqref{e:19betak},
\Be\label{e:19bky}
\int_{\R^d}|K_k(x,y)|dx\leq\int_{\R^d}|y-y_Q|\int_0^1|\nabla [\check{\beta}_k](x-\rho y-(1-\rho)y_Q)|d\rho dx\lc2^{j-n-s-k}.
\Ee
For any $x\in\R^d$, we have the following estimate
\Be\label{e:19bkx}
\begin{split}
\int_{\R^d}|K_k(x,&y)|dy\leq\sum_{Q\in\Q_{n-j+s-1}}\int_{Q}|y-y_Q|\int_0^1|\nabla [\check{\beta}_k](x-\rho y-(1-\rho)y_Q)|d\rho dy\\
&\lc2^{j-n-s-k}\int_0^1\sum_{Q\in\Q_{n-j+s-1}}2^{-kd}\int_{Q}|\nabla [{\check{\beta}}](2^{-k}(x-\rho y-(1-\rho)y_Q))| dy d\rho\\
&\lc2^{j-n-s-k}
\end{split}
\Ee
once we can show that the estimate below holds uniformly in $x,\rho,k$
\Be\label{e:19betaunif}
\sum_{Q\in\Q_{n-j+s-1}}2^{-kd}\int_{Q}|\nabla [{\check{\beta}}](2^{-k}(x-\rho y-(1-\rho)y_Q))| dy\lc 1.
\Ee

In the following we prove \eqref{e:19betaunif}. Making a change of variables $\tilde{y}=2^{-k}y$, the integral now integrates over all cubes $Q\in\Q_{n-j+s-1+k}$ with $\tilde{y}_Q=2^{-k}y_Q$ the center of this cube $Q$ which is rewritten as follows
\Bes
\begin{split}
&\sum_{Q\in \Q_{n-j+s-1+k}}\int_{Q}|\nabla [\check{\beta}](2^{-k}x-\rho \tilde{y}-(1-\rho)\tilde{y}_Q)|d\tilde{y}\\
&=\Big(\sum_{dist(Q,2^{-k}x)\leq2}+\sum_{l=1}^\infty\sum_{2^l< dist(Q,2^{-k}x)\leq2^{l+1}}\Big)\int_{Q}|\nabla [\check{\beta}](2^{-k}x-\rho \tilde{y}-(1-\rho)\tilde{y}_Q)|d\tilde{y}\\
&=:I+II,
\end{split}
\Ees
where in the second line we split the sum $\sum_{Q\in\Q_{n-j+s-1+k}}$ into two parts.
Notice that the sidelength of $Q\in \Q_{n-j+s-1+k}$ is $2^{-n+j-s+1-k}$ which is less than $1$ since we only consider the sum over $k>j-[(n+s)\eps_0]$ and $0<\eps_0\ll1$. For $I$, note that the cubes belonging in $ \Q_{n-j+s-1+k}$ are disjoint with interior, therefore the sum $\sum_{dist(Q,2^{-k}x)\leq2}$ over these cubes are supported in $B(2^{-k}x,2+\sqrt{d})$, a ball with center $2^{-k}x$ and radius $2+\sqrt{d}$. Thus we get
\Bes
|I|\lc\sum_{dist(Q,2^{-k}x)\leq2}|Q|\leq |B(2^{-k}x,2+\sqrt{d})|\leq C.
\Ees
Consider $II$. Since $\tilde{y}$ lies in a cube $Q\in\Q_{n-j+s-1+k}$ and $\tilde{y}_Q$ is the center of this cube, we get $\rho \tilde{y}+(1-\rho)\tilde{y}_Q$ lies in a line segment which is started at $\tilde{y}_Q$ and ended at $\tilde{y}$. So we have  $\rho \tilde{y}+(1-\rho)\tilde{y}_Q\in Q$ for any $\rho\in [0,1]$.
Because of $2^l<dist(Q,2^{-k}x)\leq2^{l+1}$ and $l(Q)\leq 1$,
we get $|2^{-k}x-\rho \tilde{y}-(1-\rho)\tilde{y}_Q|\approx 2^l$.
Combining the above estimates, we get
\Bes
\begin{split}
&|II|\lc\sum_{l=1}^\infty\sum_{2^l< dist(Q,2^{-k}x)\leq2^{l+1}}|Q| 2^{-(d+1)l}\lc\sum_{l=1}^\infty 2^{-l}\leq C
\end{split}
\Ees
where in the first inequality we also use the fact  $\nabla[\check{\beta}]$ is a Schwartz function which decays fast away from the origin while the second inequality follows from that the sum over all cubes $2^l< dist(Q,2^{-k}x)\leq2^{l+1}$ are supported in a ball with center $2^{-k}x$ and approximate radius $2^{l}$. Hence we finish the proof of \eqref{e:19betaunif}.

Now utilizing Schur's lemma in Lemma \ref{l:19schur} with \eqref{e:19bky} and \eqref{e:19bkx}, we get
\Bes
\|\check{\beta}_k*g_{n-j,s}\|_{L_2(\A)}\lc2^{j-n-s-k}\|g_{n-j,s}\|_{L_2(\A)}.
\Ees
Plugging this inequality into \eqref{e:19gtjcan} and later \eqref{e:19ignv} , we get
\Bes
\|(I-G_{n+s,v})\Lam_kT_j^{n,s,v}g_{n-j,s}\|_{L_2(\A)}\lc2^{j-k-(n+s)(\ga(d-1)+1-\iota)}\|\Om\|_{1}\|g_{n-j,s}\|_{L_2(\A)}.
\Ees
Taking a sum of the above estimate  over $k>j-[(n+s)\eps_0]$ yields \eqref{e:19b1kle}. Hence we complete the proof.
$\hfill{} \Box$
\vskip0.24cm

\vskip1cm

\appendix
\section{Strong $(p,p)$ bound for $\{M_r\}_{r>0}$}

\begin{theorem}
Suppose that $\Om$ satisfies \eqref{e:19Hom} and $\Om\in L_1(\S^{d-1})$. Then the operator sequences $\{M_r\}_{r>0}$ is of maximal strong type $(p,p)$ for $1<p\leq\infty$, i.e.
\Bes
\|\{M_rf\}_{r>0}\|_{L_{p}(\A,\ell_\infty(0,\infty))}\lc\|\Om\|_{1}\|f\|_{L_p(\A)}.
\Ees
\end{theorem}

\begin{proof}
By decomposing the functions $\Om$ and $f$ as four parts (i.e. real positive part, real negative part, imaginary positive part, imaginary negative part), together with triangle inequality for the norm $\|\cdot\|_{L_{p}(\A,\ell_\infty(0,\infty))}$,
we only consider the case that $\Om$ is a positive function and $f$ is positive in $\A$. Then by \eqref{e:19positivem},
it is enough to show that for any $f\in L_p^+(\A)$ there exists a positive function $F\in L_p^+(\A)$ such that
\Be\label{e:19Mrfstrong}
M_rf\leq F,\ \forall r>0 \ \ \text{and} \ \
\|F\|_{L_p(\A)}\lc\|\Om\|_1\|f\|_{L_p(\A)}.
\Ee

We will use the method of rotation. Let $f\in L_p^+(\A)$, by making a change of variables $x-y=r\tet$, we get
\Bes
\begin{split}
M_rf(x)&=\fr{1}{|B(x,r)|}\int_{B(x,r)}\Om(x-y)f(y)dy\\
&=\fr{1}{v_n}\int_{\S^{d-1}}\Om(\tet)\fr{1}{r^d}\int_0^rf(x-s\tet)s^{d-1}dsd\sigma(\tet)\\
&\lc\int_{\S^{d-1}}\Om(\tet)\Big(\fr{1}{r}\int_0^rf(x-s\tet)ds\Big)d\sigma(\tet).
\end{split}
\Ees

For a fixed $\tet\in\S^{d-1}$, we define the directional Hardy-Littlewood average operator as
\Bes
\mathfrak{M}_r^\tet f(x)=\fr{1}{r}\int_0^rf(x-s\tet)ds.
\Ees
We will prove at the end of this section the following result
\Be\label{e:19dhardyl}
\|\{\mathfrak{M}_r^\tet f\}_{r>0}\|_{L_p(\A,\ell_\infty(0,\infty)}\lc\|f\|_{L_p(\A)}.
\Ee
Assuming \eqref{e:19dhardyl} and using  \eqref{e:19positivem},
there exists a positive function $F_\tet\in L_p^+(\A)$ such that
\Bes
\mathfrak{M}_r^\tet f\leq F_\tet,\ \forall r>0 \ \ \text{and} \ \
\|F_\tet\|_{L_p(\A)}\lc \|f\|_{L_p(\A)}.
\Ees
Now if set $F(x)=\int_{\S^{d-1}}\Om(\tet)F_\tet(x)d\sigma(\tet)$, then $M_rf(x)\lc F(x)$ and
\Bes
\|F\|_{L_p(\A)}\lc\int_{\S^{d-1}}\Om(\tet)\|F_\tet\|_{L_p(\A)}d\sigma(\tet)\lc\|\Om\|_1\|f\|_{L_p(\A)}.
\Ees
Thus $F$ is the desired function satisfying \eqref{e:19Mrfstrong}.

It remains to show \eqref{e:19dhardyl}. Let $e_{1}=(1,0,\dotsm,0)$ be the unit vector. Observing that for any orthogonal matrix $A$, the following identity holds
\begin{equation}\label{rough}
\mathfrak{M}_r^{A(e_{1})}f(x)=\mathfrak{M}_r^{e_{1}}(f\circ A)(A^{-1}x).
\end{equation}
This implies that the $L_{p}$ boundedness of $\{\mathfrak{M}_r^{\theta}\}_{r>0}$ can be reduced to that of $\{\mathfrak{M}_r^{e_{1}}\}_{r>0}$.
Let $f\in L_p(L_\infty(\R^d)\overline{\otimes}\M)$. Without loss of generality, we may assume that $f$ is positive.
Fixing $x_2,...,x_d\in\R$, we consider $f(\cdot,x_2,...,x_d)$ as a function in $L_p(L_\infty(\R)\overline{\otimes}\M)_{+}$. By the strong type $(p,p)$ boundedness of noncommutative Hardy-Littlewood maximal operator (see \cite{Mei07}), we know that for $1<p\leq\infty$
$$\big\|\{\mathfrak{M}_{r}^{e_1}f(\cdot,x_{2},...,x_{d})\}_{r>0}\big\|_{L_{p}(L_\infty(\R)\overline{\otimes}\M,\ell_\infty(0,\infty))}
\lesssim\|f(\cdot,x_{2},...,x_{d})\|_{L_{p}(L_\infty(\R)\overline{\otimes}\M)}.$$
By \eqref{e:19positivem}, there exists a positive function $F(\cdot,x_2,\cdots,x_d)\in L_p(L_\infty(\R)\overline{\otimes}\M)$ such that for any $r>0$, $\mathfrak{M}^{e_1}_r f(x)\leq F(x)$ and
\begin{eqnarray*}
\|F(\cdot,x_{2},...,x_{d})\|_{L_{p}(L_\infty(\R)\overline{\otimes}\M)}\lc
\|f(\cdot,x_{2},...,x_{d})\|_{L_{p}(L_\infty(\R)\overline{\otimes}\M)}.
\end{eqnarray*}

Then it is easy to see that
$$\|F\|_{L_p(L_\infty(\R^d)\overline{\otimes}\M)}\lesssim\|f\|_{L_p(L_\infty(\R^d)\overline{\otimes}\M)}.$$

Therefore, we conclude that $\{\mathfrak{M}_r^{e_1}\}_{r>0}$ is of strong type $(p,p)$.
\end{proof}

\vskip0.4cm
\subsection*{Acknowledgement} The author would like to thank the referee for his/her very careful reading and many valuable suggestions.

\bibliographystyle{amsplain}
\bibliography{rf}

\end{document}